\documentclass[reqno,12pt]{amsart} 
\usepackage{vmargin}
\setpapersize{A4}

\usepackage{mathrsfs}

\usepackage{amsmath} 

\usepackage{amsfonts}   

\usepackage{amssymb}      

\usepackage{eufrak}      

\usepackage{amscd}      

\usepackage{amsthm}      
   
\usepackage{epsfig}      

\usepackage{amstext}      

\usepackage[all,line,dvips]{xy} 
\CompileMatrices 

\newcommand{\Ad}{\operatorname{Ad}}

\newcommand{\Span}{\operatorname{Span}}

\newcommand{\Per}{\operatorname{Per}}

 \newcommand{\supp}{\operatorname{supp}}

\newcommand{\Dim}{\operatorname{Dim}}
\newcommand{\Rank}{\operatorname{Rank}}

\newcommand{\alg}{\operatorname{alg}}

\newcommand{\inv}{^{-1}}
\newcommand{\N}{{\mathbb{N}}}
\newcommand{\Z}{{\mathbb{Z}}}
\newcommand{\T}{{\mathbb{T}}}
\newcommand{\cip}{{\mathcal{M}}}
\newcommand{\cipp}{{\cip_{\per}}}
\newcommand{\cipa}{{\cip_{\aper}}}

\DeclareMathOperator{\orb}{Orb}
\DeclareMathOperator{\per}{Per}
\DeclareMathOperator{\aper}{Aper}

   \theoremstyle{plain}
   \newtheorem{thm}{Theorem}[section]
   \newtheorem{prop}[thm]{Proposition}
   \newtheorem{lemma}[thm]{Lemma}  
   \newtheorem{cor}[thm]{Corollary}
   \theoremstyle{definition}
   
   \newtheorem{defn}[thm]{Definition}
   
   \theoremstyle{remark}

\usepackage{graphicx}

   \numberwithin{equation}{section}

        \date{\today}

\title[The $C^*$-algebra of a locally injective
surjection]{The structure of the $C^*$-algebra of a locally injective surjection}
  
\author{Toke Meier Carlsen and Klaus Thomsen}

\date{\today}

\email{tokemeie@math.ntnu.no, 
matkt@imf.au.dk}

\address{Department of Mathematics, Norwegian University of Science
  and Technology, NO-7034 Trondheim, Norway}

\address{Institut for matematiske fag, Ny Munkegade, DK-8000 Aarhus C, Denmark}

\begin{document}

\maketitle

\section{Introduction}
 
The classical connection between dynamical systems and $C^*$-algebras
is the crossed product construction which associates a $C^*$-algebra
to a ho\-meo\-mor\-phism of a compact metric space. This construction has
been generalized stepwise by J. Renault (\cite{Re}), V. Deaconu
(\cite{De1}) and C. Anantharaman-Delaroche (\cite{An}) to local
ho\-meo\-mor\-phisms and recently also to locally injective surjections by
the second named author in \cite{Th1}. The main motivation for the
last generalisation was the wish to include the Matsumoto-type
$C^*$-algebra of a subshift which was introduced by the first named
author in \cite{C}.

In this paper we continue the investigation of the structure of the
$C^*$-algebra of a locally injective surjection which was begun in \cite{Th1}. The
main goal here is to give
necessary and sufficient conditions for the algebras, or at least
any simple quotient of them, to be purely
infinite; a property they are known to have in many cases. Recall that
a simple $C^*$-algebra is said to be purely infinite when all its
non-zero hereditary $C^*$-subalgebras contain an infinite projection. Our main
result is that a simple quotient of the $C^*$-algebra arising from
a locally injective surjection on a compact
metric space of finite covering dimension, as in Section 4 of \cite{Th1}, is one of the following kinds:    
\begin{enumerate}
\item[1)] a full matrix algebra $M_n(\mathbb C)$ for some $n \in
  \mathbb N$, or
\item[2)] the crossed product $C(K) \times_f \mathbb Z$ corresponding
  to a minimal homeomorphism $f$ of a compact metric space $K$ of
  finite covering dimension, or
\item[3)] a unital purely infinite simple $C^*$-algebra.
\end{enumerate}
In particular, when the algebra itself is simple it must be one of the
three types, and in fact purely infinite unless the underlying map is
a homeomorphism. Hence the problem of finding
necessary and sufficient conditions for the $C^*$-algebra of a locally injective
surjection on a compact metric space of finite covering dimension to be both simple and purely infinite has a
strikingly straightforward solution: If the algebra is simple (and \cite{Th1}
gives necessary and sufficient conditions for this to happen) then it
is automatically purely infinite unless the map in question is a homeomorphism.
A corollary of this result is that if the $C^*$-algebra of a one-sided
subshift
is simple, then it is also purely infinite.

On the way to the proof of the main result we study the
ideal structure. We find first the gauge invariant ideals,
obtaining an insight which combined with methods and results of
Katsura (\cite{K}) leads to a list of the primitive ideals. We then
identify the maximal ideals among the primitive ones and obtain in
this way a description of the simple quotients which we use to
obtain the conclusions described above. A fundamental tool all the way
is the canonical locally homeomorphic extension discovered in
\cite{Th2} which allows us to replace the given locally injective map with a local
homeomorphism. It means, however, that much of the structure we
investigate gets described in terms of the canonical locally
homeomorphic extension, and this is unfortunate since it may not be easy
to obtain a satisfactory understanding of it for a given locally
injective surjection. Still, it allows us to obtain qualitative
conclusions of the type mentioned above.

Besides the $C^*$-algebras of subshifts our results cover of course 
also the $C^*$-algebras associated to a local homeomorphism by the
construction of Renault, Deaconu and Anantharaman-Delaroche, provided
the map is surjective and the space is of finite covering dimension. This
means that the results have bearing on many classes of $C^*$-algebras
which have been associated to various structures, for example the
$\lambda$-graph systems of Matsumoto (\cite{Ma}) and the continuous
graphs of Deaconu (\cite{De2}). 

\smallskip

\emph{Acknowledgement:} This work was supported by the NordForsk
Research Network 'Operator Algebras and Dynamics' (grant 11580). The
first named author was also supported by the Research Council of Norway
through project 191195/V30.

\section{The $C^*$-algebra of a locally injective surjection}

Let $X$ be a compact metric space and $\varphi : X \to X$ a locally
injective surjection. Set
$$
\Gamma_{\varphi} = \left\{ (x,k,y) \in X \times \mathbb Z  \times X :
  \ \exists n,m \in \mathbb N, \ k = n -m , \ \varphi^n(x) =
  \varphi^m(y)\right\} .
$$
This is a groupoid with the set of composable pairs being
$$
\Gamma_{\varphi}^{(2)} \ =  \ \left\{\left((x,k,y), (x',k',y')\right) \in \Gamma_{\varphi} \times
  \Gamma_{\varphi} : \ y = x'\right\}.
$$
The multiplication and inversion are given by 
$$
(x,k,y)(y,k',y') = (x,k+k',y') \ \text{and}  \ (x,k,y)^{-1} = (y,-k,x)
.
$$
Note that the unit space of $\Gamma_{\varphi}$ can be identified with
$X$ via the map $x \mapsto (x,0,x)$. To turn $\Gamma_{\varphi}$ into a locally compact topological groupoid, fix $k \in \mathbb Z$. For each $n \in \mathbb N$ such that
$n+k \geq 0$, set
$$
{\Gamma_{\varphi}}(k,n) = \left\{ \left(x,l, y\right) \in X \times \mathbb
  Z \times X: \ l =k, \ \varphi^{k+i}(x) = \varphi^i(y), \ i \geq n \right\} .
$$
This is a closed subset of the topological product $X \times \mathbb Z
\times X$ and hence a locally compact Hausdorff space in the relative
topology.
Since $\varphi$ is locally injective $\Gamma_{\varphi}(k,n)$ is an open subset of
$\Gamma_{\varphi}(k,n+1)$ and hence the union
$$
{\Gamma_{\varphi}}(k) = \bigcup_{n \geq -k} {\Gamma_{\varphi}}(k,n) 
$$
is a locally compact Hausdorff space in the inductive limit topology. The disjoint union
$$
\Gamma_{\varphi} = \bigcup_{k \in \mathbb Z} {\Gamma_{\varphi}}(k)
$$
is then a locally compact Hausdorff space in the topology where each
${\Gamma_{\varphi}}(k)$ is an open and closed set. In fact, as is easily verified, $\Gamma_{\varphi}$ is a locally
compact groupoid in the sense of \cite{Re} and a semi \'etale groupoid
in the sense of \cite{Th1}. The paper \cite{Th1} contains a
construction of a $C^*$-algebra from any semi \'etale groupoid,
but we give here only a description of the construction for $\Gamma_{\varphi}$.

Consider the space $B_c\left(\Gamma_{\varphi}\right)$ of
compactly supported bounded functions on $\Gamma_{\varphi}$. They form
a $*$-algebra with respect to the convolution-like product
$$
f \star g (x,k,y) = \sum_{z,n+ m = k} f(x,n,z)g(z,m,y)
$$
and the involution
$$
f^*(x,k,y) = \overline{f(y,-k,x)} .
$$
To turn it into a $C^*$-algebra, let $x \in X$ and consider the
Hilbert space $H_x$ of square summable functions on $\left\{ (x',k,y')
    \in \Gamma_{\varphi} : \ y' = x \right\}$ which carries a
  representation $\pi_x$ of the $*$-algebra $B_c\left(\Gamma_{\varphi}\right)$
  defined such that
\begin{equation}\label{pirep}
\left(\pi_x(f)\psi\right)(x',k, x) = \sum_{z, n+m = k}
f(x',n,z)\psi(z,m,x)  
\end{equation}
when $\psi  \in H_x$. One can then define a
$C^*$-algebra $B^*_r\left(\Gamma_{\varphi}\right)$ as the completion
of $B_c\left(\Gamma_{\varphi}\right)$ with respect to the norm
$$
\left\|f\right\| = \sup_{x \in X} \left\|\pi_x(f)\right\| .
$$
The space $C_c\left(\Gamma_{\varphi}\right)$ of
continuous and compactly supported functions on
$\Gamma_{\varphi}$ generate a $*$-subalgebra $\alg^*
\Gamma_{\varphi}$ of $B^*_r\left(\Gamma_{\varphi}\right)$ which completed with respect to the above norm becomes the $C^*$-algebra
$C^*_r\left(\Gamma_{\varphi}\right)$ which is our object of
study. When $\varphi$ is open, and hence a local homeomorphism,
$C_c\left(\Gamma_{\varphi}\right)$ is a $*$-subalgebra of
$B_c\left(\Gamma_{\varphi}\right)$ so that $\alg^* \Gamma_{\varphi} =
C_c\left(\Gamma_{\varphi}\right)$ and
$C^*_r\left(\Gamma_{\varphi}\right)$ is then the completion of
$C_c\left(\Gamma_{\varphi}\right)$. In this case
$C_r^*\left(\Gamma_{\varphi}\right)$ is the algebra studied by Renault
in \cite{Re}, by Deaconu in \cite{De1},
and by Anantharaman-Delaroche in \cite{An}.

The algebra $ C^*_r\left(\Gamma_{\varphi}\right)$
contains several canonical $C^*$-subalgebras which we shall need in
our study of its structure. One is the $C^*$-algebra of the open
sub-groupoid 
$$
R_{\varphi} = \Gamma_{\varphi}(0)
$$
which is a semi \'etale groupoid (equivalence relation, in fact) in itself. The corresponding
$C^*$-algebra $C^*_r\left(R_{\varphi}\right)$ is the $C^*$-subalgebra of
$C^*_r\left(\Gamma_{\varphi}\right)$ generated by the continuous and
compactly supported functions on $R_{\varphi}$. Equally important are
two canonical abelian $C^*$-subalgebras, $D_{\Gamma_{\varphi}}$ and
$D_{R_{\varphi}}$. They arise from the fact that the $C^*$-algebra
$B(X)$ of bounded functions on $X$ sits canonically inside
$B^*_r\left(\Gamma_{\varphi}\right)$, cf. p. 765 of \cite{Th1}, and are
then defined as
$$
D_{\Gamma_{\varphi}} =  C^*_r\left(\Gamma_{\varphi}\right) \cap B(X)
$$ 
and
$$
D_{R_{\varphi}} =  C^*_r\left(R_{\varphi}\right) \cap B(X),
$$ 
respectively. There are faithful conditional expectations
$P_{\Gamma_{\varphi}} : C^*_r\left(\Gamma_{\varphi}\right) \to
D_{\Gamma_{\varphi}}$ and $P_{R_{\varphi}} : C^*_r\left(R_{\varphi}\right) \to
D_{R_{\varphi}}$, obtained as extensions of the restriction map
$\alg^* \Gamma_{\varphi} \to B(X)$ to
$C^*_r\left(\Gamma_{\varphi}\right)$ and
$C^*_r\left(R_{\varphi}\right)$, respectively. When $\varphi$ is open
and hence a local homeomorphism, the two algebras
$D_{\Gamma_{\varphi}}$ and $D_{R_{\varphi}}$ are identical and equal
to $C(X)$, but in general the inclusion $D_{R_{\varphi}} \subseteq
D_{\Gamma_{\varphi}}$ is strict.

Our approach to the study of $C^*_r\left(\Gamma_{\varphi}\right)$
hinges on a construction introduced in \cite{Th2} of a compact Hausdorff space $Y$ and a local
homeomorphic surjection $\phi : Y \to Y$ such that $(X,\varphi)$ is a
factor of $(Y,\phi)$ and
\begin{equation}\label{basiciso}
C^*_r\left(\Gamma_{\varphi}\right) \simeq
C^*_r\left(\Gamma_{\phi}\right) .
\end{equation}
Everything we can say about ideals and simple quotients of
$C^*_r\left(\Gamma_{\phi}\right)$ will have bearing on
$C^*_r\left(\Gamma_{\varphi}\right)$, but while the isomorphism
(\ref{basiciso}) is equivariant with respect to the canonical gauge
actions (see Section \ref{gaugeac}), it
will not in general take $C^*_r\left(R_{\varphi}\right)$ onto
$C^*_r\left(R_{\phi}\right)$. This is one reason why we will work with
$C^*_r\left(\Gamma_{\varphi}\right)$ whenever possible, instead of
using (\ref{basiciso}) as a valid excuse for working with local
homeomorphisms only. Another is that it is generally not so easy to
get a workable description of $(Y,\phi)$. As in \cite{Th2} we will refer to
$(Y,\phi)$ as the \emph{canonical locally homeomorphic extension} of
$(X,\varphi)$. The space $Y$ is the Gelfand spectrum of
$D_{\Gamma_{\varphi}}$ so when $\varphi$ is already a local
homeomorphism itself, the extension is redundant and $(Y,\phi) = (X,\varphi)$.

\section{Ideals in $C^*_r\left(R_{\varphi}\right)$}

Recall from \cite{Th1} that there is a semi \'etale equivalence
relation 
$$
R\left(\varphi^n\right) = \left\{ (x,y) \in X \times X : \varphi^n(x)
  = \varphi^n(y) \right\}
$$ 
for each $n \in \mathbb N$. They will be considered as open
sub-equivalence relations of $R_{\varphi}$ via the embedding $(x,y)
\mapsto (x,0,y) \in \Gamma_{\varphi}(0)$. In this way we get
embeddings $C^*_r\left(R\left(\varphi^n\right)\right) \subseteq
C^*_r\left(R\left(\varphi^{n+1}\right)\right) \subseteq
C^*_r\left(R_{\varphi}\right)$ by Lemma 2.10 of \cite{Th1}, and then
\begin{equation}\label{crux}
C^*_r\left(R_{\varphi}\right) =
\overline{\bigcup_n C^*_r\left(R\left(\varphi^n\right)\right)} ,
\end{equation}
cf. (4.2) of \cite{Th1}. This inductive limit decomposition of
$C^*_r\left(R_{\varphi}\right)$ defines in a natural way a similar
inductive limit decomposition of $D_{R_{\varphi}}$. Set
$$
D_{R\left(\varphi^n\right)} =
C^*_r\left(R\left(\varphi^n\right)\right) \cap B(X) .
$$

\begin{lemma}\label{AA1} 
$D_{R_{\varphi}} =
  \overline{\bigcup_{n=1}^{\infty} D_{R\left(\varphi^n\right)}}$.
\begin{proof} Since $C^*_r\left(R\left(\varphi^n\right)\right)
  \subseteq C^*_r\left(R_{\varphi}\right)$, it follows that
$$
D_{R(\varphi^n)} = C^*_r\left(R\left(\varphi^n\right)\right) \cap
B(X) \subseteq C^*_r\left(R_{\varphi}\right) \cap B(X) = D_{R_{\varphi}}.
$$
Hence
\begin{equation}\label{AA2}
  \overline{\bigcup_{n=1}^{\infty} D_{R\left(\varphi^n\right)}}
  \subseteq D_{R_{\varphi}} .
\end{equation}
Let $x \in D_{R_{\varphi}}$ and let $\epsilon > 0$. It follows from
(\ref{crux}) that there is an $n \in \mathbb N$ and an element $y \in
\alg^* R\left(\varphi^n\right)$ such that
\begin{equation*}
\left\|x - P_{R_{\varphi}}(y)\right\| \leq \epsilon .
\end{equation*}
On $\alg^* R_{\varphi}$
the conditional expectation $P_{R_{\varphi}}$ is just the map
which restricts functions to $X$ and the same is true for the
conditional expectation
$P_{R\left(\varphi^n\right)}$ on $\alg^*
R\left(\varphi^n\right)$, where $P_{R\left(\varphi^n\right)}$ is the
conditional expectation of Lemma 2.8 in \cite{Th1} obtained by
considering $R\left(\varphi^n\right)$ as a semi \'etale groupoid in itself.  Hence $P_{R_{\varphi}}(y) =
P_{R\left(\varphi^n\right)}(y) \in
D_{R\left(\varphi^n\right)}$. It follows that we have equality in (\ref{AA2}).
\end{proof}
\end{lemma}

In the following, by an ideal of a $C^*$-algebra we will always mean a
closed and two-sided ideal. The next lemma is well known and crucial
for the sequel. 

\begin{lemma}\label{AA3} Let $Y$ be a compact Hausdorff space, $M_n$
  the $C^*$-algebra of $n$-by-$n$ matrices for some natural number $n \in
  \mathbb N$ and $p$ a projection in $C(Y,M_n)$. Set $A = pC(Y,M_n)p$
  and let $C_A$ be the center of $A$.

For every ideal $I$ in $A$ there
  is an approximate unit for $I$ in $I \cap C_A$.
\end{lemma}

\begin{lemma}\label{A7} Let $I,J \subseteq
  C^*_r\left(R_{\varphi}\right)$ be two ideals. Then
$$
I \cap D_{R_{\varphi}} \subseteq  J \cap D_{R_{\varphi}} \ \Rightarrow \ I \subseteq J.
$$ 
\begin{proof} If $I \cap D_{R_{\varphi}} \subseteq J \cap
  D_{R_{\varphi}}$ it
  follows that $I \cap D_{R\left(\varphi^n\right)} \subseteq
J \cap D_{R\left(\varphi^n\right)}$ for all $n$. Note that the center
of $C^*_r\left(R\left(\varphi^n\right)\right)$ is contained in
$D_{R\left(\varphi^n\right)}$ since $D_{R\left(\varphi^n\right)}$ is
maximal abelian in $C^*_r\left(R\left(\varphi^n\right)\right)$ by
Lemma 2.19 of \cite{Th1}. By using Corollary 3.3 of \cite{Th1} it follows therefore from Lemma \ref{AA3} that
there is a sequence $\{x_n\}$ in $ I \cap D_{R\left(\varphi^n\right)}$ such that $\lim_{n \to \infty} x_na = a$ for all $a
\in I \cap C^*_r\left(R\left(\varphi^n\right)\right)$. Since $x_n \in J \cap D_{R\left(\varphi^n\right)}$
this implies that $I \cap C^*_r\left(R\left(\varphi^n\right)\right)
\subseteq J \cap C^*_r\left(R\left(\varphi^n\right)\right)$ for all
$n$. Combining with (\ref{crux}) we find that
$$
I = \overline{\bigcup_n I \cap
  C^*_r\left(R\left(\varphi^n\right)\right)} \subseteq \overline{\bigcup_n J \cap
  C^*_r\left(R\left(\varphi^n\right)\right)} = J . 
$$
\end{proof}
\end{lemma}

Recall from \cite{Th1} that an ideal $J$ in $D_{R_{\varphi}}$ is said to
be \emph{$R_{\varphi}$-invariant} when $n^*Jn \subseteq J$ for all $n
\in \alg^* R_{\varphi}$ supported in a bisection of $R_{\varphi}$.
For every $R_{\varphi}$-invariant ideal $J$ in $D_{R_{\varphi}}$,
set
$$
\widehat{J} = \left\{ a \in C^*_r\left(R_{\varphi}\right) : \
  P_{R_{\varphi}}(a^*a) \in J \right\} .
$$

\begin{thm}\label{A4} The map $J \mapsto \widehat{J}$ is a bijection
  between the $R_{\varphi}$-invariant ideals in
  $D_{R_{\varphi}}$ and the ideals in
$C^*_r\left(R_{\varphi}\right)$. The inverse is given by the map $I
\mapsto I \cap D_{R_{\varphi}}$
\begin{proof} It follows from Lemma 2.13 of \cite{Th1} that $\widehat{J} \cap
  D_{R_{\varphi}} = J$ for any $R_{\varphi}$-invariant ideal in
  $D_{R_{\varphi}}$. It suffices therefore to show that every ideal in
  $C^*_r\left(R_{\varphi}\right)$ is of the form $\widehat{J}$ for some
  $R_{\varphi}$-invariant ideal $J$ in
  $D_{R_{\varphi}}$.
  Let $I$ be an ideal in
  $C^*_r\left(R_{\varphi}\right)$. Set $J = I \cap
  D_{R_{\varphi}}$, which is clearly a $R_{\varphi}$-invariant ideal in
  $D_{R_{\varphi}}$. Since $\widehat{J} \cap D_{R_{\varphi}} = J
  = I \cap D_{R_{\varphi}}$ by Lemma 2.13 of \cite{Th1}, 
  we conclude from Lemma \ref{A7} that $\widehat{J} = I$.

\end{proof}
\end{thm}

A subset $A \subseteq Y$ is said to be \emph{$\phi$-saturated} when
$\phi^{-k}\left(\phi^k(A)\right) = A$ for all $k \in \mathbb N$.

\begin{cor}\label{A5} (Cf. Proposition II.4.6 of \cite{Re}) The map 
$$
L \mapsto I_L = \left\{ a \in C^*_r\left(R_{\phi}\right) : \
  P_{R_{\phi}}(a^*a)(x) = 0\ \forall
  x \in L \right\}
$$
is a bijection from the non-empty closed $\phi$-saturated subsets $L$ of $Y$ onto the set of proper ideals in
$C^*_r\left(R_{\phi}\right)$.
\begin{proof} Since $\phi$ is a local homeomorphism, we have that $D_{R_{\phi}} =
  C(Y)$ so the corollary follows from Theorem \ref{A4} by use of the
  well-known bijection between ideals in $C(Y)$ and closed subsets of
  $Y$. The only thing to show is that an open subset $U$ of $Y$ is 
  $\phi$-saturated if and only if the ideal $C_0(U)$ of $C(Y)$ is
  $R_{\phi}$-invariant which is straightforward, cf. the proof of
  Corollary 2.18 in \cite{Th1}.
\end{proof}
\end{cor} 


The next issue will be to determine which closed $\phi$-saturated
subsets of $Y$ correspond to primitive ideals. For a point $x \in Y$
we define the \emph{$\phi$-saturation} of $x$ to be the set
$$
H(x) = \bigcup_{n=1}^{\infty} \left\{ y \in Y : \ \phi^n(y) =
  \phi^n(x) \right\} .
$$
The closure $\overline{H(x)}$ of $H(x)$ will be referred to as the
\emph{closed $\phi$-saturation} of $x$. Observe that both $H(x)$ and
$\overline{H(x)}$ are $\phi$-saturated.

\begin{prop}\label{A17} Let $L \subseteq Y$ be a non-empty closed
  $\phi$-saturated subset. The ideal $I_L$ is primitive if and only
  $L$ is the closed $\phi$-saturation of a point in $Y$. 

\begin{proof} Since $C^*_r\left(R_{\phi}\right)$ is separable an ideal
  is primitive if and only if it is prime, cf. \cite{Pe}. We show that $I_L$ is prime
  if and only if $L =   \overline{H(x)}$ for some $x \in Y$. 

Assume
  first that $L = \overline{H(x)}$ and consider two ideals, $I_1$ and
  $I_2$, in $C^*_r\left(R_{\phi}\right)$ such that $I_1I_2 \subseteq
  I_{\overline{H(x)}}$. By Corollary \ref{A5} there are closed
  $\phi$-saturated subsets, $L_1$ and $L_2$, in $Y$ such that $I_j =
  I_{L_j}$, $j =1,2$. It follows from Corollary \ref{A5} that $\overline{H(x)} \subseteq L_1
  \cup L_2$. At least one of the $L_j$'s must contain $x$, say $x \in
  L_1$. Since $L_1$ is $\phi$-saturated and closed it follows that
  $\overline{H(x)} \subseteq L_1$, and hence that $I_1 \subseteq
  I_{\overline{H(x)}}$. Thus $I_{\overline{H(x)}}$ is prime.

Assume next that $I_L$ is prime. Let $\{U_k\}_{k=0}^{\infty}$ be
a base for the topology of $L$ consisting of non-empty sets. We will construct sequences
$\{B_k\}_{k=0}^{\infty}$ of compact non-empty
neighbourhoods in $L$ and non-negative integers
$\left\{n_k\right\}_{k=0}^{\infty}$ such that
\begin{enumerate}
\item[i)] $B_k \subseteq B_{k-1}$ for $ k \geq 1$, and
\item[ii)] $\phi^{n_{k}}\left(B_{k}\right) \subseteq
  \phi^{n_{k}}\left(U_{k}\right)$ for $ k \geq 0$.
\end{enumerate}
We start the induction by letting $B_0$ be any compact non-empty neighbourhood in
$U_0$ and $n_0 = 0$. Assume then that $B_0,B_1,B_2,\dots ,
B_m$ and $n_0,n_1, \dots, n_m$ have been
constructed. Choose a non-empty open subset $V_{m+1} \subseteq B_{m}$. Note that both of
$$
L  \backslash \bigcup_l \phi^{-l}\left(\phi^l(V_{m+1})\right) 
$$
and
$$
L  \backslash \bigcup_l \phi^{-l}\left(\phi^l(U_{m+1})\right)   
$$
are closed $\phi$-saturated subsets of $L$, and hence of $Y$, and none of them is all
of $L$. It follows from Corollary \ref{A5} and primeness of
$I_L$ that $L$ is not contained in their union, which in turn implies
that
$$
\phi^{-n_{m+1}}\left(\phi^{n_{m+1}}(V_{m+1})\right) \cap
  \phi^{-n_{m+1}}\left(\phi^{n_{m+1}}(U_{m+1})\right)
$$
is non-empty for some $n_{m+1} \in  \mathbb N$. There is therefore a
point $z \in V_{m+1}$ such that $\phi^{n_{m+1}}(z) \in
\phi^{n_{m+1}}\left(U_{m+1}\right)$, and therefore also a compact
non-empty neighbourhood $B_{m+1} \subseteq V_{m+1}$ of $z$ such that $\phi^{n_{m+1}}(B_{m+1}) \subseteq
\phi^{n_{m+1}}\left(U_{m+1}\right)$. This completes the induction. Let
$x \in \bigcap_m B_m$. By construction every $U_k$ contains an element
from $H(x)$. It follows that $\overline{H(x)}
= L$.  
\end{proof}
\end{prop}

\section{On the ideals of $C^*_r\left(\Gamma_{\varphi}\right)$}\label{gaugeac}

The $C^*$-algebra $C^*_r\left(\Gamma_{\varphi}\right)$ carries a
canonical circle action $\beta$, called the \emph{gauge action}, defined such that
$$
\beta_{\lambda}(f)(x,k,y) = \lambda^k f(x,k,y)
$$
when $f \in C_c\left(\Gamma_{\varphi}\right)$ and $\lambda \in \mathbb
T$, cf. \cite{Th1}. As the next step we describe in this section the
gauge-invariant ideals in $C^*_r\left(\Gamma_{\varphi}\right)$.

Consider first the function $m : X \to \mathbb N$ defined such that
\begin{equation}\label{m-funk}
m(x) = \# \left\{ y \in X : \ \varphi(y) = \varphi(x) \right\} .
\end{equation}
As shown in \cite{Th1}, $m \in D_{R(\varphi)} \subseteq D_{R_{\varphi}}$. Define a function $V_{\varphi} : \Gamma_{\varphi} \to \mathbb C$ such that
$$
V_{\varphi} (x,k,y) = \begin{cases} m(x)^{-\frac{1}{2}} & \ \text{when} \ k = 1 \
  \text{and} \ y = \varphi(x) \\ 0 & \ \text{otherwise.} \end{cases}
$$
Then $V_{\varphi}$ is the product $V_{\varphi} = m^{-\frac{1}{2}} 1_{\Gamma_{\varphi}(1,0)}$ 
in $C^*_r\left(\Gamma_{\varphi}\right)$ and in fact an isometry which
induces an endomorphism $\widehat{\varphi}$ of
$C^*_r\left(R_{\varphi}\right)$, viz.
$$
\widehat{\varphi}(a) = V_{\varphi}aV_{\varphi}^*$$
Together with $C^*_r\left(R_{\varphi}\right)$ the isometry
$V_{\varphi}$ generates $C^*_r\left(\Gamma_{\varphi}\right)$ which in
this way becomes a crossed product $C^*_r\left(R_{\varphi}\right)
\times_{\widehat{\varphi}} \mathbb N$ in the sense of Stacey,
cf. \cite{St} and  \cite{Th1}; in particular Theorem 4.6 in \cite{Th1}.

\subsection{Gauge invariant ideals }

Let $C^*_r\left(\Gamma_{\varphi}\right)^{\mathbb T}$ denote the fixed
point algebra of the gauge action.

\begin{lemma}\label{kalgs} For each $k \in \mathbb N$ we have that ${V_{\varphi}^*}^k
  C^*_r\left(R_{\varphi}\right)V_{\varphi}^k $ is a $C^*$-subalgebra of
  $C^*_r\left(\Gamma_{\varphi}\right)^{\mathbb T}$,
\begin{equation}\label{bkr0}
{V_{\varphi}^*}^kC^*_r\left(R_{\varphi}\right)V_{\varphi}^k \subseteq
  {V_{\varphi}^*}^{k+1}C^*_r\left(R_{\varphi}\right)V_{\varphi}^{k+1},
\end{equation} 
and
\begin{equation}\label{bkr}
C^*_r\left(\Gamma_{\varphi}\right)^{\mathbb T} =
\overline{\bigcup_{k=0}^{\infty} {V_{\varphi}^*}^k
  C^*_r\left(R_{\varphi}\right)V_{\varphi}^k }.
\end{equation}
\begin{proof} Since $V_{\varphi}^k{V_{\varphi}^*}^k \in C^*_r\left(R_{\varphi}\right)$, it
  is easy to check that ${V_{\varphi}^*}^k
  C^*_r\left(R_{\varphi}\right)V_{\varphi}^k$ is a $*$-algebra. To see that
  ${V_{\varphi}^*}^kC^*_r\left(R_{\varphi}\right)V_{\varphi}^k$ is closed let $\{a_n\}$ be a sequence
  in $C^*_r\left(R_{\varphi}\right)$ such that
  $\left\{{V_{\varphi}^*}^ka_nV_{\varphi}^k\right\}$ converges in
  $C^*_r\left(\Gamma_{\varphi}\right)$, say $\lim_{n \to \infty}
  {V_{\varphi}^*}^ka_nV_{\varphi}^k = b$. It follows that
  $$
\left\{V_{\varphi}^k{V_{\varphi}^*}^ka_nV_{\varphi}^k{V_{\varphi}^*}^k\right\}
$$ 
is Cauchy in
  $C^*_r\left(R_{\varphi}\right)$ and hence convergent, say to $a \in
  C^*_r\left(R_{\varphi}\right)$. But then $b = \lim_{n \to \infty}
  {V_{\varphi}^*}^k a_nV_{\varphi}^k = \lim_{n \to \infty}
 {V_{\varphi}^*}^kV_{\varphi}^k {V_{\varphi}^*}^k a_nV_{\varphi}^k{V_{\varphi}^*}^kV_{\varphi}^k = {V_{\varphi}^*}^kaV_{\varphi}^k$. It follows that
 $$
{V_{\varphi}^*}^kC^*_r\left(R_{\varphi}\right)V_{\varphi}^k
$$ 
is closed and hence a
 $C^*$-subalgebra. The inclusion (\ref{bkr0}) follows from the observation that $V_{\varphi}^k = V_{\varphi}^*V_{\varphi}^{k+1}$ and $V_{\varphi}C^*_r\left(R_{\varphi}\right)V_{\varphi}^*
\subseteq C^*_r\left(R_{\varphi}\right)$.

It is straightforward to check that $\beta_{\lambda}(V_{\varphi}) =
  \lambda V_{\varphi}$ and that $C^*_r\left(R_{\varphi}\right) \subseteq
  C^*_r\left(\Gamma_{\varphi}\right)^{\mathbb T}$. The inclusion
  $\supseteq$ in (\ref{bkr}) follows from this. To obtain the other, let $x
  \in C^*_r\left(\Gamma_{\varphi}\right)^{\mathbb T}$ and let
$\epsilon > 0$. It follows from Theorem 4.6 of \cite{Th1} and Lemma
1.1. of \cite{BoKR} that there is an $n \in \mathbb N$ and an element 
$$
y \in \Span \bigcup_{i,j \leq n} {V_{\varphi}^*}^i
C^*_r\left(R_{\varphi}\right)V_{\varphi}^j 
$$
such that $\left\|x - y\right\| \leq \epsilon$. Then $\left\| x -
  \int_{\mathbb T} \beta_{\lambda}(y) \ d\lambda\right\| \leq
\epsilon$ and since
$$
\int_{\mathbb T} \beta_{\lambda}(y) \ d\lambda \in
{V_{\varphi}^*}^nC^*_r\left(R_{\varphi}\right)V_{\varphi}^n,
$$
we see that (\ref{bkr}) holds.
\end{proof}
\end{lemma}

\begin{lemma}\label{ident} Let $I$ be a gauge invariant ideal in
  $C^*_r\left(\Gamma_{\varphi}\right)$. It follows that
$$
I = \left\{ a \in C^*_r\left(\Gamma_{\varphi}\right) : \ \int_{\mathbb
    T} \beta_{\lambda}(a^*a) \ d \lambda \in I \cap
  C^*_r\left(\Gamma_{\varphi}\right)^{\mathbb T} \right\} .
$$
\begin{proof} Set $B = C^*_r\left(\Gamma_{\varphi}\right)/I$. Since
  $I$ is gauge-invariant there is an action $\hat{\beta}$ of $\mathbb
  T$ on $B$ such that $q \circ \beta = \hat{\beta} \circ q$, where $q
  : C^*_r\left(\Gamma_{\varphi}\right) \to B$ is the quotient
  map. Thus, if 
$$
y \in \left\{ a \in C^*_r\left(\Gamma_{\varphi}\right) : \ \int_{\mathbb
    T} \beta_{\lambda}(a^*a) \ d \lambda \in I \cap
  C^*_r\left(\Gamma_{\varphi}\right)^{\mathbb T} \right\} ,
$$
we find that 
$$
\int_{\mathbb T} \hat{\beta}_{\lambda}(q(y^*y)) \ d \lambda =
q\left(\int_{\mathbb T} \beta_{\lambda}(y^*y) \ d\lambda \right) = 0 .
$$
Since $\int_{\mathbb T} \hat{\beta}_{\lambda}( \cdot ) \ d \lambda$ is
faithful we conclude that $q(y) = 0$, i.e. $y \in I$. This establishes the non-trivial part of the asserted identity.
\end{proof}
\end{lemma}

\begin{lemma}\label{intersectunique} Let $I,I'$ be gauge invariant ideals in $C^*_r\left(\Gamma_{\varphi}\right)$. Then
$$
I \cap D_{R_{\varphi}} \subseteq I' \cap D_{R_{\varphi}} \ \Rightarrow \ I \subseteq
I' .
$$
\begin{proof} Assume that $I \cap D_{R_{\varphi}} \subseteq I' \cap
  D_{R_{\varphi}}$. It follows from Lemma \ref{A7} that $I \cap
  C^*_r\left(R_{\varphi}\right) \subseteq  I' \cap
C^*_r\left(R_{\varphi}\right)$. Then
\begin{equation*}
\begin{split}
&I \cap
  {V_{\varphi}^*}^kC^*_r\left(R_{\varphi}\right)V_{\varphi}^k = {V_{\varphi}^*}^k\left(I \cap
  C^*_r\left(R_{\varphi}\right)\right)V_{\varphi}^k \\
& \ \ \ \ \ \ \ \ \ \ \ \ \ \ \subseteq {V_{\varphi}^*}^k\left(I' \cap
  C^*_r\left(R_{\varphi}\right)\right)V_{\varphi}^k = I' \cap
  {V_{\varphi}^*}^kC^*_r\left(R_{\varphi}\right)V_{\varphi}^k
\end{split}
\end{equation*}
for all $k \in \mathbb N$. Hence Lemma \ref{kalgs} implies
that $I \cap C_r^*\left(\Gamma_{\varphi}\right)^{\mathbb T} \subseteq I' \cap
C_r^*\left(\Gamma_{\varphi}\right)^{\mathbb T}$. It follows then from Lemma \ref{ident} that $I \subseteq I'$.

\end{proof}
\end{lemma}

\begin{prop}\label{gaugeideals} The map $J \mapsto \widehat{J}$, where
$$
\widehat{J} = \left\{ a \in C^*_r\left(\Gamma_{\varphi}\right) : \
  P_{\Gamma_{\varphi}}(a^*a) \in J\right\},
$$ 
is a
  bijection from the $\Gamma_{\varphi}$-invariant ideals of
  $D_{\Gamma_{\varphi}}$ onto the gauge invariant ideals of
  $C^*_r\left(\Gamma_{\varphi}\right)$. Its inverse is the map $I
  \mapsto I \cap D_{\Gamma_{\varphi}}$.
\begin{proof} Since $P_{\Gamma_{\varphi}} \circ \beta =  P_{\Gamma_{\varphi}}$ the ideal $\widehat{J}$ is gauge invariant. It
  follows from Lemma 2.13 of \cite{Th1} that $\widehat{J} \cap D_{\Gamma_{\varphi}} = J$
so it suffices to show that
\begin{equation}\label{gaugeb}
\widehat{I \cap D_{\Gamma_{\varphi}}} = I
\end{equation}
when $I$ is a gauge invariant ideal in
$C^*_r\left(\Gamma_{\varphi}\right)$. It follows from Lemma 2.13 of
\cite{Th1} that
$\widehat{I \cap D_{\Gamma_{\varphi}}} \cap D_{\Gamma_{\varphi}} = I
\cap D_{\Gamma_{\varphi}}$.
Since $D_{R_{\varphi}} \subseteq D_{\Gamma_{\varphi}}$ this implies that
$\widehat{I \cap D_{\Gamma_{\varphi}}} \cap D_{R_{\varphi}} = I \cap
D_{R_{\varphi}}$.
Then (\ref{gaugeb}) follows from Lemma \ref{intersectunique}.

\end{proof}
\end{prop}

To simplify notation, set $D = D_{\Gamma_{\phi}} = C(Y)$. Every
ideal $I$ in $C^*_r\left(\Gamma_{\phi}\right)$ determines a closed
subset $\rho(I)$ of $Y$ defined such that
\begin{equation}\label{rho}
\rho(I) = \left\{ y \in Y : \ f(y) = 0 \ \forall f \in
  I \cap D \right\} .
\end{equation}

We say that a subset $F \subseteq Y$ is
\emph{totally $\phi$-invariant} when $\phi^{-1}(F) = F$.

\begin{lemma}\label{psiinv} $\rho(I)$ is totally $\phi$-invariant for
  every ideal $I$ in $C^*_r\left(\Gamma_{\phi}\right)$.
\begin{proof} It suffices to show that $Y\setminus
  \rho(I)$ is totally
  $\phi$-invariant, which is what we will do. Assume first that $x\in
  Y \setminus \rho(I)$. Then there is an $f\in
  I\cap D$ such that $f(x)\neq 0$. 
  Choose an open bisection $W \subseteq \Gamma_{\phi}$ such that
  $(x,1,\phi(x)) \in W$. Choose then $\eta\in
  C_c(\Gamma_\phi)$ such that $\eta((x,1,\phi(x))=1$ and $\supp \eta
  \subseteq W$. It is not
  difficult to check that $\eta^*f\eta\in D$ and that
  $\eta^*f\eta(\phi(x))= f(x)\ne 0$, and since 
  $\eta^*f\eta\in I$, it follows that
  $\phi(x)\in Y \setminus \rho(I)$. Assume then that $\phi(x)\in Y
  \setminus \rho(I)$. Then there is a
  $g\in I\cap D$ such that $g(\phi(x))\ne 0$. Choose an open bisection $W \subseteq \Gamma_{\phi}$ such that
  $(x,1,\phi(x)) \in W$ and $\eta\in
  C_c(\Gamma_\phi)$ such that $\eta((x,1,\phi(x))=1$ and $\supp \eta
  \subseteq W$. Then $\eta g\eta^*\in D$ and 
  $\eta g \eta^*(\phi(x))= g(x)\ne 0$, and since 
  $\eta g\eta^*\in I$, this shows that
  $x \in Y\backslash \rho(I)$, proving that
  $\phi^{-1}\left(\rho(I)\right) = \rho(I)$.
\end{proof}
\end{lemma}

Thus every ideal in $C^*_r\left(\Gamma_{\phi}\right)$ gives rise to
a closed totally $\phi$-invariant subset of $Y$. To go in
the other direction, let $F$ be a closed totally $\phi$-invariant
subset of $Y$. Then $Y\backslash F$ is open and
totally $\phi$-invariant so that the reduction
$\Gamma_{\phi}|_{Y \backslash F}$ is an \'etale groupoid in
its own right, cf. \cite{An}. In fact, $\phi$ restricts to local homeomorphic
surjections $\phi : Y \backslash F \to Y \backslash F$ and $\phi : F
\to F$, and 
$$
\Gamma_{\phi}|_{Y \backslash F} =  \Gamma_{\phi|_{Y \backslash F}} .
$$
Note that $C^*_r\left( \Gamma_{\phi}|_{Y \backslash
    F}\right) = C^*_r\left( \Gamma_{\phi|_{Y \backslash
    F}}\right)$  is an ideal in $C^*_r\left(\Gamma_{\phi}\right)$
because $Y \backslash F$ is totally $\phi$-invariant.

\begin{prop} \label{prop:canonic} (Cf. Proposition II.4.5 of \cite{Re}.) 
  Let $F$ be a non-empty, closed and totally $\phi$-invariant subset
  of $Y$. There is then a
  surjective $*$-homomorphism $\pi_F:  C_r^*(\Gamma_\phi)\to
  C_r^*(\Gamma_{\phi|_F})$ which extends the restriction map
   $C_c\left(\Gamma_{\phi}\right) \to C_c\left(\Gamma_{\phi|_F}\right)$
  and has the property that $\ker \pi_F =
  C^*_r\left(\Gamma_{\phi|_{Y \backslash F}}\right)$, i.e.
\begin{equation*}
\begin{xymatrix}{
0 \ar[r] & C^*_r\left(\Gamma_{\phi|_{Y \backslash F}}\right) \ar[r] &
C_r^*(\Gamma_\phi) \ar[r]^-{\pi_F} \ar[r] &  C_r^*(\Gamma_{\phi|_F})
\ar[r] & 0}
\end{xymatrix}
\end{equation*}
is exact. Furthermore,
\begin{equation}\label{rhoF}
\rho(\ker\pi_F)=F.
\end{equation}
\end{prop}

\begin{proof} Let $\dot{\pi_F} : C_c\left(\Gamma_{\phi}\right) \to
  C_c\left(\Gamma_{\phi|_F}\right)$ denote the restriction map which
  is surjective by Tietze's theorem. By using that $F$ is totally
  $\phi$-invariant, it follows straightforwardly that $\dot{\pi_F}$ is
  a $*$-homomorphism. Since $\pi_x \circ \dot{\pi_F} = \pi_x$ when $x
  \in F$, it follows that $\dot{\pi_F}$ extends by continuity to a
  $*$-homomorphism $\pi_F :  C_r^*(\Gamma_\phi)\to
  C_r^*(\Gamma_{\phi|_F})$ which is surjective because $\dot{\pi_F}$
  is. To complete the proof observe that
\begin{equation*}\label{estblis}
\ker \pi_F \cap D = C_0\left(Y \backslash F\right) = C^*_r\left(\Gamma_{\phi|_{Y \backslash
        F}}\right) \cap D .
\end{equation*}
The first identity shows that (\ref{rhoF}) holds, and since $\ker \pi_F$
and $C^*_r\left(\Gamma_{\phi|_{Y \backslash
        F}}\right)$ are both gauge-invariant ideals the second that
  they are identical by Lemma \ref{intersectunique}.
\end{proof}

By combining Proposition \ref{gaugeideals}, Lemma \ref{psiinv} and Proposition
\ref{prop:canonic} we obtain the following.

\begin{thm}\label{psi-invariant} The map $\rho$ is a bijection
  from the gauge-invariant ideals in
  $C^*_r\left(\Gamma_{\phi}\right)$ onto the set of closed totally $\phi$-invariant
  subsets of $Y$. The inverse is the map which sends a closed totally
  $\phi$-invariant subset $F \subseteq Y$ to the ideal
$$
\ker \pi_F = \left\{ a \in C^*_r\left(\Gamma_{\phi}\right) : \
  P_{\Gamma_{\phi}}(a^*a)(x) = 0 \ \forall x \in F \right\} .
$$
\end{thm}

We remark that since the isomorphism (\ref{basiciso}) is equivariant
with respect to the gauge actions, Theorem \ref{psi-invariant} gives
also a description of the gauge invariant ideals in
$C^*_r\left(\Gamma_{\varphi}\right)$, as a complement to the one of
Proposition \ref{gaugeideals}.

\subsection{The primitive ideals}

We are now in position to obtain a complete description of the
primitive ideals of $C^*_r\left(\Gamma_{\phi}\right)$. Much of what we
do
is merely a translation of Katsuras description of the primitive
ideals in the more general $C^*$-algebras considered by him in
\cite{K}. Recall that because we only deal with separable
$C^*$-algebras the primitive ideals are the same as the prime ideals,
cf. 3.13.10 and 4.3.6 in \cite{Pe}.

\begin{lemma} \label{prop:ideal-gen}
  Let $I$ be an ideal in $C_r^*(\Gamma_\phi)$ and let
  $A$ be a closed totally $\phi$-invariant subset of $Y$.
  If $\rho(I)\subseteq A$, then $\ker\pi_A\subseteq I$.
\end{lemma}

\begin{proof}
  Since $\rho(I)\subseteq A$ it follows from
  the Stone-Weierstrass theorem that $C_0(Y\setminus A)\subseteq I
  \cap C(Y)$. Let
  $\left\{i_n\right\}$ be an approximate unit in $C_0(Y \backslash
  A)$. It follows from Proposition \ref{prop:canonic} that $\{i_n\}$ is
  also an approximate unit in $\ker \pi_A$. Since $\{i_n\} \subseteq
  I$ it follows that $\ker\pi_A\subseteq I$.  
\end{proof}

We say that a closed totally $\phi$-invariant subset $A$ of $Y$ is \emph{prime}
  when it has the property that if $B$ and $C$ also are closed and
  totally $\phi$-invariant subsets of $Y$ and $A\subseteq B\cup C$, then either
  $A\subseteq B$ or $A\subseteq C$.

Let 
$\cip:=\{A\subseteq Y :  A\text{ is non-empty, closed, totally $\phi$-invariant and
  prime}\}$. 
For $x\in Y$ let 
$$
\orb(x) =\{y\in Y :  \exists
m,n\in\N:\phi^n(x)=\phi^m(y)\}.
$$ 
We call $\orb(x)$ the \emph{total
  $\phi$-orbit of $x$}.

\begin{prop}(Cf. Proposition 4.13 and 4.4 of \cite{K}.)\label{glemt}
  \begin{equation*}
    \cip=\{\overline{\orb(x)} :  x\in Y\}.
  \end{equation*}
\end{prop}

\begin{proof}
  It is clear that $\overline{\orb(x)}\in\cip$ for every $x\in Y$. Assume that $A\in\cip$ and let $\{U_k\}_{k=1}^\infty$ be a basis for
  $A$. We will by induction show that we can choose
  compact neighbourhoods $\{C_k\}_{k=0}^\infty$ and
  $\{C_k'\}_{k=0}^\infty$ in $A$ and positive integers $(n_k)_{k=0}^\infty$
  and $(n_k')_{k=0}^\infty$ such that $C_k\subseteq U_k$ and
  $C_k'\subseteq \phi^{n_{k-1}}(C_{k-1})\cap
  \phi^{n'_{k-1}}(C'_{k-1})$ for $k\ge 1$. For this set $C_0=C'_0=A$. Assume then that $n\ge 1$ and that
  $C_1,\dots,C_n$, $C'_1,\dots,C'_n$, $n_0,\dots,n_{n-1}$ and $n'_0,\dots,n'_{n-1}$
  satisfying the conditions above have been chosen. Choose non-empty
  open subsets $V_n\subseteq C_n$ and $V'_n\subseteq C'_n$.
  We then have that 
  \begin{equation*}
    \bigcup_{l,m=0}^\infty\phi^{-l}(\phi^m(V_n))\text{ and }\bigcup_{l,m=0}^\infty\phi^{-l}(\phi^m(V'_n))
  \end{equation*}
  are non-empty open and totally $\phi$-invariant
  subsets of $A$, and thus that 
  \begin{equation} \label{eq:1}
    A\setminus\bigcup_{l,m=0}^\infty\phi^{-l}(\phi^m(V_n))\text{ and
    }A\setminus\bigcup_{l,m=0}^\infty\phi^{-l}(\phi^m(V'_n)) 
  \end{equation}
  are closed, totally $\phi$-invariant subsets of $Y$. Since $A$ is prime and is not
  contained in either of the sets from \eqref{eq:1}, it follows that
  $A$ is not contained in 
  \begin{equation*} 
    \left(A\setminus\bigcup_{l,m=0}^\infty\phi^{-l}(\phi^m(V_n))\right)\bigcup
    \left(A\setminus\bigcup_{l,m=0}^\infty\phi^{-l}(\phi^m(V'_n)) \right),
  \end{equation*}
 whence
  \begin{equation*} 
    \left(\bigcup_{l,m=0}^\infty\phi^{-l}(\phi^m(V_n))\right)\bigcap
    \left(\bigcup_{l,m=0}^\infty\phi^{-l}(\phi^m(V'_n))\right) \ne\emptyset.
  \end{equation*}
  It follows that there are positive integers $n_n$ and $n'_n$
  such that $\phi^{n_n}(V_n)\cap\phi^{n'_n}(V'_n)$ is
  non-empty. Thus we can choose a compact neighbourhood $C_{n+1}\subseteq
  U_{n+1}$ and a compact neighbourhood $C'_{n+1}\subseteq
  \phi^{n_n}(V_n)\cap\phi^{n'_n}(V'_n)\subseteq
  \phi^{n_n}(C_n)\cap\phi^{n'_n}(C'_n)$ which is what is required for the induction step.

  It is easy to check that
  \begin{equation*}
    C'_0\cap\phi^{-n'_0}(C'_1)\cap\dots \dots \cap\phi^{-n'_0- \dots -n'_k}(C'_{k+1}),\ k=0,1,\dots
  \end{equation*}
  is a decreasing sequence of non-empty compact sets. It follows
  that there is an 
  \begin{equation*}
    x\in \bigcap_{k=0}^\infty \phi^{-n'_0-\dots \dots -n'_k}(C'_{k+1})\cap C'_0.
  \end{equation*}
  We have for every $k\in\N$ that $\phi^{n'_0+\dots+n'_k}(x)\in
  C'_{k+1}\subseteq \phi^{n_k}(C_k)\subseteq \phi^{n_k}(U_k)$, and it
  follows that $\orb(x)$ is dense in $A$, and thus that $A=\overline{\orb(x)}$.
\end{proof}

\begin{prop}(Cf. Proposition 9.3 of \cite{K}.) \label{prop:prime}
Assume that $I$ is a prime ideal in $C_r^*(\Gamma_\phi)$. It follows that $\rho(I)\in\cip$.
\end{prop}

\begin{proof}
  It follows from Lemma \ref{psiinv} that $\rho(I)$ is closed and
  totally $\phi$-invariant.
  To show that $\rho(I)$ is also prime, assume that $B$ and $C$ are
  closed totally $\phi$-invariant subsets such that $\rho(I)\subseteq
  B\cup C$.  
  It follows then from
  Lemma \ref{prop:ideal-gen} that $\ker(\pi_{B\cup C})\subseteq
  I$. Since $\ker \pi_B \cap \ker \pi_C \cap D =
  C_0(Y \backslash B) \cap C_0(Y \backslash C) = C_0\left(Y \backslash
  (B\cup C)\right) =  \ker \pi_{B \cup C} \cap D$ it follows from
Lemma \ref{intersectunique} that $\ker \pi_B \cap \ker \pi_C = \ker
\pi_{B\cup C}$. Therefore $\ker(\pi_B)\subseteq I$ or
  $\ker(\pi_C)\subseteq I$ since $I$ is prime. Hence $\rho(I)\subseteq B$ or
  $\rho(I)\subseteq C$, thanks to (\ref{rhoF}).
\end{proof}

We say that a point $x\in Y$ is $\phi$-periodic if $\phi^n(x)=x$
for some $n>0$. Let $\per $ denote the set of $\phi$-periodic points $x\in
Y$ which are isolated in $\orb(x)$ and let
$$
\cipp =\{\overline{\orb(x)} :  x\in \per\}
$$ 
and 
$$
\cipa =\cip\setminus\cipp.
$$
Let $A \subseteq Y$ be a closed totally $\phi$-invariant subset. We say that
$\phi|_A$ is \emph{topologically free} if the set of $\phi$-periodic
points in $A$ has empty interior in $A$.

\begin{prop} (Cf. Proposition 11.3 of \cite{K}.) \label{prop:free}
  Let $A\in\cip$. Then $\phi|_A$ is topologically free if and only
  if $A\in\cipa$. 
\end{prop}
 
\begin{proof}
  We will show that $\phi|_A$ is not topologically free if and only
  if $A\in\cipp$. If $x\in\per$ and $A=\overline{\orb(x)}$, then $\phi|_A$ is not
  topologically free because $x$ is periodic and isolated in $\orb(x)$
  and thus in $A$. Assume then that $\phi|_A$ is not topologically free. There is then
  a non-empty open subset $U\subseteq A$ such that every element of $U$ is
  $\phi$-periodic. Choose $x\in A$ such that
  $A=\overline{\orb(x)}$. Then $U\cap\orb(x)\ne\emptyset$. Let $y\in
  U\cap\orb(x)$. Then $y$ is $\phi$-periodic and
  $\overline{\orb(y)}=\overline{\orb(x)}=A$, so if we can show that
  $y$ is isolated in $\orb(y)$, then we have that $A\in\cipp$. 
  Since $y$ is $\phi$-periodic there is
  an $n\ge 1$ such that $\phi^n(y)=y$. We claim that
  $U\subseteq\{y,\phi(y),\dots,\phi^{n-1}(y)\}$. It will then follow that $y$
  is isolated in $A$ and thus in $\orb(y)$.

  Assume that $U\setminus \{y,\phi(y),\dots,\phi^{n-1}(y)\}$ is
  non-empty. Since it is also open it follows that $\orb(y)\cap
  U\setminus \{y,\phi(y),\dots,\phi^{n-1}(y)\}$ is non-empty. Let
  $z\in \orb(y)\cap
  U\setminus \{y,\phi(y),\dots,\phi^{n-1}(y)\}$. Since $z\in U$ there
  is an $m\ge 1$ so that $\phi^m(z)=z$, and since $z\in\orb(y)$ there
  are $k,l\in\N$ such that $\phi^k(z)=\phi^l(y)$. But then
  $z=\phi^{mk}(z)=\phi^{(m-1)k+l}(y)\in
  \{y,\phi(y),\dots,\phi^{n-1}(y)\}$ and we have a contradiction. It
  follows that $U\subseteq\{y,\phi(y),\dots,\phi^{n-1}(y)\}$.
\end{proof}

In particular, it follows from Proposition \ref{prop:free} that the
elements of $\cipa$ are infinite sets.

\begin{prop} (Cf. Proposition 11.5 of \cite{K}.) \label{prop:aper}
  Let $A\in\cipa$. Then $\ker\pi_A$ is the unique ideal $I$ in
  $C_r^*(\Gamma_\phi)$ with $\rho(I)=A$.
\end{prop}

\begin{proof}
  We have already in Proposition \ref{prop:canonic} seen that
  $\rho(\ker\pi_A)=A$. Assume that $I$ is an ideal in
  $C_r^*(\Gamma_\phi)$ with $\rho(I)=A$. It follows then from
  Lemma \ref{prop:ideal-gen} that $\ker\pi_A\subseteq I$. Thus
  it sufficies to show that $\pi_A(I)=\{0\}$. Note that $\pi_A(I)$ is an ideal in
  $C_r^*(\Gamma_{\phi|_A})$ with $\rho(\pi_A(I))=A$. It
  follows that $\pi_A(I)\cap C(A)=\{0\}$. To conclude from this that
  $\pi_A(I) = \{0\}$ we will show that the points of $A$
  whose isotropy group in $\Gamma_{\phi|_A}$ is trivial are dense in
  $A$. It will then follow from
  Lemma 2.15 of \cite{Th1} that $\pi_A(I)=\{0\}$ because $\pi_A(I) \cap
  C(A) = \{0\}$. 
  That the points of $A$ with trivial isotropy in $\Gamma_{\phi|_A}$ are dense in
  $A$ is established as follows: The points in $A$ with
  non-trivial isotropy in $\Gamma_{\phi|A}$ are the pre-periodic
  points in $A$. Let $\Per_n A$ denote the set of points in $A$ with
  minimal period $n$ under $\phi$ and note that $\Per_n A$ is closed and
  has empty interior since
  $\phi|_A$ is topologically free by Proposition \ref{prop:free}. It follows that
  $A \backslash \phi^{-k}\left(\Per_n A\right)$ is open and dense in $A$
  for all $k,n$. By the Baire category theorem it follows that
  $$
  A \backslash \left( \bigcup_{k,n} \phi^{-k}\left(\Per_n A\right) \right)
  = \bigcap_{k,n} A \backslash \phi^{-k}\left(\Per_n A\right)
  $$
  is dense in $A$, proving the claim.

\end{proof}

\begin{lemma}\label{prop2}
Let $A\in\cipa$. Then $\ker\pi_A$ is a primitive ideal.
\begin{proof} Let $A = \overline{\orb(x)}$. To show that $\ker \pi_A$ is primitive it suffices to show that it is
prime, cf. Proposition 4.3.6 of \cite{Pe}. Equivalently, it suffices
to show that $C^*_r\left(\Gamma_{\phi|_A}\right)$ is a prime
$C^*$-algebra. Consider therefore two ideals $I_j \subseteq
C^*_r\left(\Gamma_{\phi|_A}\right), j = 1,2$, such that $I_1I_2 =
\{0\}$. Then
$$
\left\{ y \in A : \ f(y) = 0 \ \forall f \in I_1 \cap C(A) \right\}
\cup \left\{ y \in A : \ f(y) = 0 \ \forall f \in I_2 \cap C(A)
\right\} = A .
$$ 
In particular, $x$ must be in $\left\{ y \in A : \ f(y) = 0 \ \forall
  f \in I_j \cap C(A) \right\}$, either for $j = 1$ or $j =2$. It
follows then from Lemma \ref{psiinv}, applied to $\phi|_A$, that 
$$
A
= \left\{ y \in A : \ f(y) = 0 \ \forall f \in I_j \cap C(A)
\right\}.
$$ 
Hence $I_j = \{0\}$ by Proposition \ref{prop:aper} applied
to $\phi|_A$.
\end{proof}
\end{lemma}

Let $A\in\cipp$. Choose $x\in\per$ such that $\overline{\orb(x)}=A$,
and let $n$ be the minimal period of $x$. Then $x$ is isolated in
$A$. It follows that the characteristic functions $1_{(x,0,x)}$ and $1_{(x,n,x)}$
belong to $C_r^*(\Gamma_{\phi|_A})$. Let
$p_x=1_{(x,0,x)}$ and $u_x=1_{(x,n,x)}$. For $w\in\T$ let
$\dot{P}_{x,w}$ denote the ideal in  $C_r^*(\Gamma_{\phi|_A})$
generated by $u_x-wp_x$.

\begin{lemma} \label{remark:ens}
  Suppose that $x,y\in\per$ and that
  $\overline{\orb(x)}=\overline{\orb(y)}$ and let $w\in\T$. Then
  $\dot{P}_{x,w}=\dot{P}_{y,w}$. 
\end{lemma}
\begin{proof}
  By symmetry, it is enough to show that
  $\dot{P}_{y,w}\subseteq\dot{P}_{x,w}$. 
  Since $y$ is isolated in $\orb(y)$, it is isolated in
  $\overline{\orb(y)}=\overline{\orb(x)}$. Thus $y$ must belong to
  $\orb(x)$. This means that there are $k,l\in\N$ such that
  $\phi^k(x)=\phi^l(y)$. Since $y$ is $\phi$-periodic, it follows that
  there is an $i\in\N$ such that $y=\phi^i(x)$. Let
  $A=\overline{\orb(y)}=\overline{\orb(x)}$. Since $x$ and $y$ are
  isolated in $A$ we have that
  $1_{(x,i,y)}\in C_r^*(\Gamma_{\phi|_A})$. Let $v=1_{(x,i,y)} $. It
  is easy to check that $v^*p_xv=p_y$ and that $v^*u_xv=u_y$. Thus
  $u_y-wp_y=v^*(u_x-wp_x)v\in \dot{P}_{x,w}$ and it follows that
  $\dot{P}_{y,w}\subseteq\dot{P}_{x,w}$.  
\end{proof}

Let $A\in\cipp$ and let $w\in\T$. It follows from Lemma \ref{remark:ens} that the ideal $\dot{P}_{x,w}$
does not depend of the particular choice of $x \in A \cap \Per$, as long as
$\overline{\orb(x)}=A$. We will therefore simply write $\dot{P}_{A,w}$
for $\dot{P}_{x,w}$. We then define $P_{A,w}$ to be the ideal
$\pi_A\inv(\dot{P}_{A,w})$ in $C_r^*(\Gamma_\phi)$.

\begin{prop} (Cf. Proposition 11.13 of \cite{K}.) \label{prop:per}
  Let $A\in\cipp$. Then
  \begin{equation*}
    w\mapsto P_{A,w}
  \end{equation*}
  is a bijection between $\T$ and the set of primitive ideals $P$ in
  $C_r^*(\Gamma_\phi)$ with $\rho(P)=A$.
\end{prop}

\begin{proof}
  The map $P\mapsto\pi_A(P)$ gives a bijection between the primitive
  ideals in $C_r^*(\Gamma_\phi)$ with $\ker\pi_A\subseteq P$ and the
  primitive ideals in $C_r^*(\Gamma_{\phi|_A})$, cf. Theorem 4.1.11
  (ii) in \cite{Pe}. The inverse of this
  bijection is the map $Q\mapsto \pi_A\inv(Q)$.  
  If $P$ is a primitive ideal in $C_r^*(\Gamma_\phi)$ with $\rho(P)=A$, it
  follows from Lemma \ref{prop:ideal-gen} that
  $\ker\pi_A\subseteq P$. In addition $\rho(\pi_A(P))=A$. If on the other
  hand $Q$ is a primitive ideal in $C_r^*(\Gamma_{\phi|_A})$ with
  $\rho(Q)=A$, then $\pi_A\inv(Q)$ is a primitive ideal in
  $C_r^*(\Gamma_\phi)$ and $\rho(\pi_A\inv(Q))=A$. Thus
  $P\mapsto\pi_A(P)$ gives a bijection between the primitive ideals in
  $C_r^*(\Gamma_\phi)$ with $\rho(P)=A$ and the
  primitive ideals $Q$ in $C_r^*(\Gamma_{\phi|_A})$ with $\rho(Q)=A$.
  
  Choose $x\in\per$ such that $\overline{\orb(x)}=A$.
  Let $\langle p_x\rangle$ be the ideal in $C_r^*(\Gamma_{\phi|_A})$  generated by
  $p_x$. Observe that $\dot{P}_{A,w} \subseteq \langle p_x \rangle$
  for all $w \in \mathbb T$ since $p_x\left(u_x -wp_x\right) = u_x - wp_x$.
  The map $Q\mapsto Q\cap\langle p_x\rangle$ gives a bijection between
  the primitive ideals in $C_r^*(\Gamma_{\phi|_A})$ with $\langle
  p_x\rangle\nsubseteq Q$ and the primitive ideals in $\langle
  p_x\rangle$, cf. Theorem 4.1.11 (ii) in \cite{Pe}. We claim that
  $\langle p_x\rangle\nsubseteq Q$ if and only if $\rho(Q)=A$. To see
  this, let $Q$ be an ideal in $C_r^*(\Gamma_{\phi|_A})$. If $p_x\in Q$,
  then $x\notin \rho(Q)$ and $\rho(Q)\ne A$. If on the other hand
  $\rho(Q)\ne A$, then $x\notin \rho(Q)$ because $\rho(Q)$ is closed
  and totally $\phi$-invariant and $\overline{\orb(x)}=A$. It follows that there
  is an $f\in Q\cap C(A)$ such that $f(x)\ne 0$, whence $p_x\in
  Q$. This proves the claim and it follows that $Q\mapsto Q\cap\langle p_x\rangle$ gives a bijection between
  the primitive ideals in $C_r^*(\Gamma_{\phi|_A})$ with $\rho(Q)=A$ and the primitive ideals in $\langle
  p_x\rangle$.

  The $C^*$-algebra $\langle p_x\rangle$ is Morita equivalent to
  $p_xC_r^*(\Gamma_{\phi|_A})p_x$ via the
  $p_xC_r^*(\Gamma_{\phi|_A})p_x$- $\langle p_x\rangle$ imprimitivity
  bimodule $p_xC_r^*(\Gamma_{\phi|_A})$, and therefore $T\mapsto p_xTp_x$ gives a
  bijection between the primitive ideals $T$ in $\langle p_x\rangle$ and
  the primitive ideals in $p_xC_r^*(\Gamma_{\phi|A})p_x$,
  cf. Proposition 3.24 and Corollary 3.33 in \cite{RW}. Now note that 
  \begin{equation*}
    \{(x',n',y') \in\Gamma_{\phi|_A} : 
    x'=y' =x \}=\{(x,kn,x) :  k\in\Z\}
  \end{equation*}
  where $n$ is the smallest positive integer such that
  $\phi^n(x)=x$. It follows that $p_xC_r^*(\Gamma_{\phi|A})p_x$ is
  isomorphic to $C(\T)$ under an isomorphism taking the canonical
  unitary generator of $C(\mathbb T)$ to $u_x$. In this way we
  conclude that the primitive ideals of $p_xC_r^*(\Gamma_{\phi|A})p_x$
  are in one-to-one correspondance with $\mathbb T$ under the map
$$
\mathbb T \ni w \mapsto p_x\overline{C_r^*(\Gamma_{\phi|A})
\left(u_x-wp_x\right)C_r^*(\Gamma_{\phi|A})}p_x = p_x \dot{P}_{A,w} p_x.
$$
This completes the proof.
\end{proof}

By combining Proposition \ref{prop:prime}, \ref{prop:aper} and
\ref{prop:per} we get the following theorem.

\begin{thm}\label{primitive}
  The set of primitive ideals in $C_r^*(\Gamma_\phi)$ is the disjoint
  union of $\{\ker\pi_A :  A\in\cipa\}$ and $\{P_{A,w} :  A\in\cipp,\ w\in\T\}$.
\end{thm}

\subsection{The maximal ideals}

The next step is to identify the maximal ideals among the primitive ones.

\begin{lemma}\label{intcx} Assume that not all points of $Y$ are
  pre-periodic and that $C^*_r\left(\Gamma_{\phi}\right)$ contains a
  non-trivial ideal.  It follows that there is a non-trivial
  gauge-invariant ideal $J$ in
  $C^*_r\left(\Gamma_{\phi}\right)$ such that $J \cap C(Y) \neq \{0\}$.
 \begin{proof} Let $I$ be a non-trivial ideal in
  $C^*_r\left(\Gamma_{\phi}\right)$. Assume first that $I \cap C(Y) =
  \{0\}$. Since we assume that not all points of $Y$ are pre-periodic
  we can apply Lemma 2.16 of \cite{Th1} to conclude that $J_0 = \overline{P_{\Gamma_{\phi}}(I)}$ is a
  non-trivial $\Gamma_{\phi}$-invariant ideal in $C(Y)$. Then
$$
J = \left\{ a \in C^*_r\left(\Gamma_{\phi}\right) : \
  P_{\Gamma_{\phi}}(a^*a) \in J_0 \right\}
$$
is a non-trivial gauge-invariant ideal by Theorem
\ref{gaugeideals}, and $J \cap C(Y) = J_0 \neq \{0\}$. Note that $J$ contains $I$ in this case. If $I \cap C(X) \neq \{0\}$ we set 
$$
J = \left\{ a \in C^*_r\left(\Gamma_{\phi}\right) : \
  P_{\Gamma_{\phi}}(a^*a) \in I \cap C(Y) \right\}
$$
which is a non-trivial ideal in $C^*_r\left(\Gamma_{\phi}\right)$ such
that $J \cap C(Y) = I \cap C(Y)$ by
Lemma 2.13 of \cite{Th1}. Since $J$ is gauge-invariant, this completes
the proof.
\end{proof}
\end{lemma}

\begin{lemma}\label{minimal} Let $F \subseteq Y$ be a minimal closed
  non-empty totally $\phi$-invariant subset. Then either
\begin{enumerate}
\item[1)] $F \in \mathcal M_{Aper}$ and $\ker \pi_F$ is a maximal
  ideal, or
\item[2)] $F = \orb(x) = \left\{ \phi^n(x) : n  \in \mathbb N\right\}$, where $x \in \Per$.
\end{enumerate}
\begin{proof} It follows from the minimality of $F$ that $\overline{\orb(x)}=F$
  for all $x\in F$. We will show that 1) holds when $F$ does not
  contain an element of $\per$, and that 2) holds when it does. Assume first that $F$ does not contain any elements of $\per$. Then
  $F\in \mathcal M_{Aper}$. If  there is a proper ideal $I$ in $C_r^*(\Gamma_\phi)$ such that $\ker \pi_F
  \subsetneq I$, then $\pi_F(I)$ is a non-trivial ideal in
  $C^*_r\left(\Gamma_{\phi|_F}\right)$, and then it follows from
  Lemma \ref{intcx} that there is a non-trivial gauge-invariant ideal $J$ in
  $C^*_r\left(\Gamma_{\phi|_F}\right)$. By Theorem
  \ref{psi-invariant} $\rho(\pi_F^{-1}(J))$ is then a
  non-trivial closed totally $\phi$-invariant subset of
 $F$, contradicting the minimality of $F$. Thus 1) holds when $F$ does not
  contain an element from $\per$.

 Assume instead that there is an $x\in F\cap\Per$.
  Then $x$ is isolated in $\orb(x)$, and thus in
  $F$. It follows that $F=\orb(x)$, because if $y\in
  F\setminus\orb(x)$ we would have that $x\notin
  \overline{\orb(y)}=F$, which is absurd.
  Since $F$ is compact, $\orb(x)$ must be finite. Since $\phi$ is surjective we
  must then have that 
  $\orb(x) =  \left\{ \phi^n(x) : n  \in \mathbb N\right\}$.
  Thus 2) holds if $F$ contains
  an element from $\Per$.
\end{proof}
\end{lemma}

\begin{lemma}\label{maxideal1} Let $I$ be a maximal ideal in
  $C^*_r\left(\Gamma_{\phi}\right)$. Then either $I = \ker \pi_F$ for
  some minimal closed totally $\phi$-invariant subset $F \in \mathcal M_{Aper}$, or $I =
  P_{\orb(x),w}$ for some $w \in \mathbb T$ and some $x \in \Per$ such
  that $\orb(x) =
\left\{\phi^n(x): \ n \in \mathbb N\right\}$.
\begin{proof} Since $I$ is also primitive we know from Theorem
  \ref{primitive} that $I = \ker \pi_A$ for some $A \in \mathcal
  M_{Aper}$ or $I = P_{A,w}$ for some $A \in \mathcal M_{\per}$ and
  some $w \in \mathbb T$. In the first case it follows that $A$ must
  be a minimal closed totally $\phi$-invariant subset since $I$ is a maximal
  ideal. Assume then that $I = P_{A,w}$ for some $A \in \mathcal M_{\per}$ and
  some $w \in \mathbb T$. In the notation from the proof of Proposition
  \ref{prop:per}, observe that $\dot{P}_{A,w} \subseteq \langle p_x
  \rangle$ since $p_x\left(u_x-wp_x\right) = u_x -wp_x$. Note that $\dot{P}_{A,w} \neq \langle p_x
  \rangle$ because the latter of these ideals is gauge-invariant and the
  first is not. By maximality of $I$ this implies that $\langle p_x
  \rangle = C^*_r\left(\Gamma_{\phi|_A}\right)$. On the other hand,
  $\orb(x)$ is an open totally $\phi$-invariant subset of $A$ and $p_x \in
  C^*_r\left( \Gamma_{\phi|_{\orb(x)}}\right)$, so we see that $\langle p_x
  \rangle = C^*_r\left(\Gamma_{\phi|_A}\right) =
  C^*_r\left( \Gamma_{\phi|_{\orb(x)}}\right)$. This implies that
$$
C_0\left(\orb(x)\right) = C(A) \cap C^*_r\left(
  \Gamma_{\phi|_{\orb(x)}}\right) = C(A),
$$
and hence that $A = \orb(x)$. Compactness of $A$ implies that
$\orb(x)$ is finite and surjectivity of $\phi$ that $\orb(x) =
\left\{\phi^n(x): \ n \in \mathbb N\right\}$.
\end{proof}
\end{lemma}

\begin{thm}\label{maximal2} The maximal ideals in
  $C^*_r\left(\Gamma_{\phi}\right)$ consist of the primitive ideals of
  the form $\ker \pi_F$ for some infinite minimal closed totally $\phi$-invariant subset $F
  \subseteq Y$ and the primitive ideals $P_{A,w}$ for some $w \in \T$, where $A = \orb(x) =
  \left\{\phi^n(x) : \ n \in \mathbb N\right\}$ for a $\phi$-periodic point
  $x \in Y$.
\begin{proof} This follows from the last two lemmas, after the
  observation that a primitive ideal $P_{A,w}$ of the form described
  in the statement is maximal.
\end{proof}
\end{thm}

\begin{cor}\label{maximal3} Let $A$ be a simple quotient of
  $C^*_r\left(\Gamma_{\phi}\right)$. Assume $A$ is not finite dimensional. It follows that there is an infinite minimal closed totally $\phi$-invariant
  subset $F$ of $Y$ such that $A \simeq C^*_r\left(
    \Gamma_{\phi|_F}\right)$.
\end{cor}

To make more detailed conclusions about the simple quotients we
need to restrict to the case where $Y$ is of finite covering dimension so that
the result of \cite{Th3} applies. For this reason we prove first that
finite dimensionality of $Y$ follows from finite dimensionality of $X$.

\section{On the dimension of $Y$}

Let $\Dim X$ and $\Dim Y$ denote the covering dimensions of $X$ and
$Y$, respectively.
The purpose with this section is to establish

 \begin{prop}\label{dim!!!} $\Dim Y \leq \Dim X$.\end{prop}
\begin{proof}By definition $Y$ is the Gelfand spectrum of
  $D_{\Gamma_{\varphi}}$. Since the conditional expectation
  $P_{\Gamma_{\varphi}} : C^*_r\left(\Gamma_{\varphi}\right) \to
  D_{\Gamma_{\varphi}}$ is invariant under the gauge action, in the
  sense that $P_{\Gamma_{\varphi}} \circ \beta_{\lambda} =
  P_{\Gamma_{\varphi}}$ for all $\lambda$, it follows that
\begin{equation*}\label{D17}
D_{\Gamma_{\varphi}} =
P_{\Gamma_{\varphi}}\left(C^*_r\left(\Gamma_{\varphi}\right)^{\mathbb
    T}\right) .
\end{equation*}
To make use of this description of $D_{\Gamma_{\varphi}}$ we need a refined version of
(\ref{bkr}). Note first that it follows from (4.4) and (4.5) of
\cite{Th1} that
$V_{\varphi}C^*_r\left(R\left(\varphi^l\right)\right)V_{\varphi}^*
\subseteq C^*_r\left(R\left(\varphi^{l+1}\right)\right)$ for all $l \in
  \mathbb N$. Consequently
$$
{V_{\varphi}^*}^k C^*_r\left(R\left(\varphi^l\right)\right)
V_{\varphi}^k = {V_{\varphi}^*}^{k+1}V_{\varphi} C^*_r\left(R\left(\varphi^l\right)\right)V_{\varphi}^*
V_{\varphi}^{k+1} \subseteq {V_{\varphi}^*}^{k+1} C^*_r\left(R\left(\varphi^{l+1}\right)\right)
V_{\varphi}^{k+1}
$$
for all $k,l  \in \mathbb N$. It follows therefore from (\ref{crux})
and (\ref{bkr}) that there are sequences $\{k_n\}$ and
$\left\{l_n\right\}$ in $\mathbb N$ such that $l_n \geq k_n$,
\begin{equation}\label{D17}
{V_{\varphi}^*}^{k_n}
C^*_r\left(R\left(\varphi^{l_n}\right)\right)V_{\varphi}^{k_n}
\subseteq {V_{\varphi}^*}^{k_{n+1}} C^*_r\left(R\left(\varphi^{l_{n+1}}\right)\right)V_{\varphi}^{k_{n+1}}
\end{equation}
and
\begin{equation}\label{D1}
C^*_r\left(\Gamma_{\varphi}\right)^{\mathbb
    T} = \overline{ \bigcup_n {V_{\varphi}^*}^{k_n}
C^*_r\left(R\left(\varphi^{l_n}\right)\right)V_{\varphi}^{k_n}}; 
\end{equation}
we can for example use $k_n=n$ and $l_n=2n$.

Let $D_n$ denote the $C^*$-subalgebra of
  $D_{\Gamma_{\varphi}}$ generated by
$$
P_{\Gamma_{\varphi}}\left({V^*_{\varphi}}^{k_n}
  C^*_r\left(R\left(\varphi^{l_n} \right)\right)V_{\varphi}^{k_n}\right) 
$$
and let $Y_n$ be the character space of $D_n$.
Note that $C(X) \subseteq D_n$ since
$V_{\varphi}^{k_n}g{V_{\varphi}^*}^{k_n} \in
C^*_r\left(R\left(\varphi^{l_n}\right)\right)$ and $g = P_{\Gamma_{\varphi}}\left( {V^*_{\varphi}}^{k_n} V_{\varphi}^{k_n}g 
  {V_{\varphi}^*}^{k_n} V^{k_n}_{\varphi} \right)$ when $g \in C(X)$. There is therefore a continuous surjection 
$$
\pi_n :
Y_n \to X
$$ 
defined such that $g\left(\pi_n(y)\right) = y(g), \ g \in
C(X)$. We claim that $\# \pi_n^{-1}(x) < \infty$ for all $x \in X$. To
show this note that by
definition $D_n$ is generated as a $C^*$-algebra by functions of the
form
\begin{equation}\label{expresssion}
\begin{split}
&x \mapsto P_{\Gamma_{\varphi}}\left( {V^*_{\varphi}}^{k_n}
  fV_{\varphi}^{k_n}\right)(x)  = \sum_{z,z' \in \varphi^{-k_n}(x)} 
 f(z,z') \prod_{j=0}^{k_n-1} m(\varphi^j(z))^{-\frac{1}{2}} m(\varphi^j(z'))^{-\frac{1}{2}}
\end{split}
\end{equation}
for some $f \in C^*_r\left(R\left(\varphi^{l_n}\right)\right)$. In fact, since $\alg^* R\left(\varphi^{l_n}\right)$ is dense in
$C^*_r\left(R\left(\varphi^{l_n}\right)\right)$, already functions of
the form (\ref{expresssion})
with
\begin{equation}\label{expression7}
f = f_1 \star f_2 \star \dots \star
f_N, 
\end{equation}
for some $f_i \in C\left(R\left(\varphi^{l_n}\right)\right), i =
1,2,\dots, N$, will generate $D_n$.


Fix $x \in X$ and consider an element $y \in \pi_n^{-1}(x)$. Every $x'
\in X$ defines a character $\iota_{x'}$ of $D_n$ by
evaluation, viz. $\iota_{x'}(h) = h(x')$, and $\left\{\iota_{x'} : x'
  \in X \right\}$ is dense in $Y_n$ because the implication
$$
h \in D_n, \ h(x') = 0
\ \forall x' \in X \ \Rightarrow \ h = 0
$$
holds. 
In particular, there is a
sequence $\left\{x_l\right\}$ in $X$ such that $\lim_{l \to \infty}
\iota_{x_l} = y$ in $Y_n$. Recall now from Lemma 3.6 of \cite{Th1} that there is an
  open neighbourhood $U$ of $x$ and open sets $V_j, j=1,2, \dots,d$, where $d = \# \varphi^{-k_n}(x)$, in
  $X$ such
  that
\begin{enumerate}
\item[1)] $\varphi^{-k_n}\left(\overline{U}\right) \subseteq V_1 \cup V_2
  \cup \dots \cup V_d$,
\item[2)] $\overline{V_i} \cap \overline{V_j} = \emptyset, \ i \neq
  j$, and
\item[3)] $\varphi^{k_n}$ is injective on $\overline{V_j}$ for
  each $j$.
\end{enumerate} 
Since $\lim_{l \to \infty} x_l = x$ in $X$ we can assume that $x_l \in
U$ for all $l$. For each $l$, set
$$
F_l = \left\{ j : \ \varphi^{-k_n}(x_l) \cap V_j \neq \emptyset
\right\} \subseteq \left\{1,2,\dots,d\right\}.
$$
Note that there is a subset $F \subseteq \left\{1,2,\dots,d\right\}$
such that $F_l = F$ for infinitely many $l$. Passing to a subsequence
we can therefore assume that $F_l = F$ for all $l$. For each $k \in F$
we define a
continuous map $\lambda_k :
\varphi^{k_n}\left(\overline{V_k}\right) \to \overline{V_k}$ such that
$\varphi^{k_n} \circ \lambda_k (z) = z$. Set $T = \max_{z
  \in X} \# \varphi^{-1}(z)$. For each $j \in \left\{1,2, \dots,
  T\right\}$, set
$$
A_j = \left\{ z \in X : \# \varphi^{-1}\left(\varphi(z)\right)
  = j \right\} = m^{-1}(j). 
$$
For each $l$ and each $k \in F$ there is a unique tuple $\left(j_0(k),j_1(k),
\dots, j_{k_n-1}(k)\right) \in \left\{1,2, \dots, T\right\}^{k_n}$ such
that
$$
\varphi^{-k_n}(x_l) \cap V_k \cap A_{j_0(k)} \cap
\varphi^{-1}\left(A_{j_1(k)}\right) \cap \varphi^{-2}\left(A_{j_2(k)}\right)
\cap \dots \cap \varphi^{-k_n+1 }\left(A_{j_{k_n-1}(k)}\right) \neq \emptyset
.
$$ 
Since there are only finitely many choices we can arrange that the
same tuples, $\left(j_0(k),j_1(k),
\dots, j_{k_n-1}(k)\right), k  \in F$, work for all $l$. Then
\begin{equation}\label{expresssion2}
\begin{split}
&\iota_{x_l}\left(P_{\Gamma_{\varphi}}\left( {V^*_{\varphi}}^{k_n}
  fV_{\varphi}^{k_n}\right)\right)  = \sum_{k,k' \in F} 
 f\left(\lambda_k(x_l),\lambda_{k'}(x_l)\right) \prod_{i=0}^{k_n-1} j_i(k)^{-\frac{1}{2}} j_i(k')^{-\frac{1}{2}}
\end{split}
\end{equation}
for all $f \in C^*_r\left(R\left(\varphi^{l_n}\right)\right)$ and all
$l$. 

There is an open neighbourhood $U'$ of $\varphi^{l_n-k_n}(x)$ and open sets $V'_j, j=1,2, \dots,d'$, where $d' = \# \varphi^{-l_n}\left(\varphi^{l_n-k_n}(x)\right)$, in
  $X$ such
  that
\begin{enumerate}
\item[1')] $\varphi^{-l_n}\left(\overline{U'}\right) \subseteq V'_1 \cup V'_2
  \cup \dots \cup V'_{d'}$,
\item[2')] $\overline{V'_i} \cap \overline{V'_j} = \emptyset, \ i \neq
  j$, and
\item[3')] $\varphi^{l_n}$ is injective on $\overline{V'_j}$ for
  each $j$.
\end{enumerate} 
Since $\lim_{l \to \infty} \varphi^{l_n-k_n}(x_l) = \varphi^{l_n
  -k_n}(x)$ we can assume that $\varphi^{l_n -k_n}(x_l) \in U'$ for
all $l$. By an argument identical to the way we found $F$ above we can
now find a subset $F' \subseteq \{1,2,\dots, d'\}$ such that
$$
F' = \left\{ j : \varphi^{-l_n}\left(\varphi^{l_n-k_n}(x_l)\right)
  \cap V'_j \neq \emptyset \right\}
$$
for all $l$. For $i \in F'$ we define a continuous map $\mu'_i :
\varphi^{l_n}\left(\overline{V'_i}\right) \to \overline{V'_i}$ such
that $\mu'_i \circ \varphi^{l_n}(z) = z$ when $z \in
\overline{V'_i}$. Set
$$
\mu_i = \mu'_i \circ \varphi^{l_n-k_n} 
$$  
on $\varphi^{-(l_n-k_n)}\left(\varphi^{l_n}\left(\overline{V'_i}\right)\right)$.
Assuming that $f$ has the form (\ref{expression7}) we find now that
\begin{equation}\label{yrk}
\begin{split}
&f\left(\lambda_k(x_l),\lambda_{k'}(x_l)\right) = \\
& \sum_{i_1,i_2, \dots,
  i_{N-1} \in F'} f_1\left(\lambda_k(x_l),
  \mu_{i_1}(x_l)\right)f_2\left(\mu_{i_1}(x_l),
  \mu_{i_2}(x_l)\right)\dots \dots
f_N\left(\mu_{i_{N-1}}(x_l), \lambda_{k'}(x_l)\right)  
\end{split}
\end{equation}
for all $k,k' \in F$. By combining (\ref{yrk}) with
(\ref{expresssion2}) we find by letting $l$ tend to infinity that 
$$
y\left(P_{\Gamma_{\varphi}}\left( {V^*_{\varphi}}^{k_n}
  fV_{\varphi}^{k_n}\right)\right)  = \sum_{k,k' \in F}
H_{k,k'}(x)\prod_{i=0}^{k_n-1} j_i(k)^{-\frac{1}{2}}
j_i(k')^{-\frac{1}{2}},
$$
where 
$$
H_{k,k'}(x) = \sum_{i_1,i_2, \dots,
  i_{N-1} \in F'} f_1\left(\lambda_k(x),
  \mu_{i_1}(x)\right)f_2\left(\mu_{i_1}(x),
  \mu_{i_2}(x)\right)\dots \dots
f_N\left(\mu_{i_{N-1}}(x), \lambda_{k'}(x)\right) .
$$
Since this expression only depends on $F,F'$ and the tuples 
$$
\left(j_0(k),j_1(k),
\dots, j_{k_n-1}(k)\right), k  \in F,
$$ 
it follows that the number of possible values of an element from $\pi_n^{-1}(x)$
on the generators of the form (\ref{expresssion}) does not exceed
$2^d2^{d'}T^{k_n}$, proving that
$\# \pi_n^{-1}(x) < \infty$ as claimed.

We can then apply Theorem 4.3.6 on page 281 of \cite{En} to conclude
that $\Dim
Y_n \leq \Dim X$. Note that $D_n \subseteq D_{n+1}$ and  
$D_{\Gamma_{\varphi}} =
\overline{\bigcup_n D_n}$ by (\ref{D17}) and
(\ref{D1}). Hence $Y$ is the projective
limit of the sequence $Y_1 \gets Y_2 \gets Y_3 \gets \dots $. Since $\Dim Y_n \leq \Dim X$ for all $n$ we conclude
now from Theorem 1.13.4 in \cite{En} that $\Dim Y \leq \Dim X$. 
\end{proof}

\section{The simple quotients}

Following \cite{DS} we say that $\phi$ is \emph{strongly transitive}
when for any non-empty open subset $U
\subseteq Y$ there is an $n \in \mathbb N$ such that $Y =
\bigcup_{j=0}^n \phi^j(U)$, cf. \cite{DS}. By Proposition 4.3 of
\cite{DS}, $C^*_r\left(\Gamma_{\phi}\right)$ is simple if and only if
$Y$ is infinite and 
$\phi$ is strongly transitive.

\begin{lemma}\label{hmzero} Assume that $\phi$ is strongly transitive
  but not injective. It
  follows that 
$$
\lim_{k \to \infty} \frac{1}{k} \log \left(\inf_{x \in Y}\# \phi^{-k}(x)\right) >
0.
$$ 

\begin{proof} Note that 
$U = \left\{ x \in Y : \ \# \phi^{-1}(x) \geq 2 \right\}$
is open and not empty since $\phi$ is a local homeomorphism and not
injective. It follows that there is an $m \in \mathbb N$ such that 
\begin{equation}\label{D18}
\bigcup_{j=0}^{m-1} \phi^j(U) = Y 
\end{equation}
because $\phi$ is strongly transitive. We claim that
\begin{equation}\label{est1}
\inf_{z \in Y} \# \phi^{-k}(z)  \geq 2^{\left[\frac{k}{m}\right]}
\end{equation}
for all $k \in \mathbb N$ where $\left[\frac{k}{m}\right]$ denotes the
integer part of $\frac{k}{m}$. This follows by induction: Assume that it true for
all $k' < k$. Consider any $z \in Y$. If $k < m$ there is nothing to
prove so assume that $k \geq m$. By (\ref{D18}) we can then write $z =
\phi^j(z_1) = \phi^j(z_2)$ for some $j \in   \left\{1,2,\dots,
  m\right\}$ and some $z_1 \neq z_2$. It follows that 
$$
\# \phi^{-k}(z) \geq \# \phi^{-(k-j)}(z_1) + \# \phi^{-(k-j)}(z_2)
\geq 2 \cdot 2^{\left[\frac{k-j}{m}\right]} \geq
2^{\left[\frac{k}{m}\right]} .
$$ 
It follows from (\ref{est1}) that $\lim_{k \to \infty} \frac{1}{k}
\log \left(\inf_{x \in Y}\# \phi^{-k}(x)\right) \geq \frac{1}{m} \log 2$.
\end{proof}
\end{lemma}

Let $M_l$ denote the $C^*$-algebra of complex $l \times l$-matrices.
In the following a \emph{homogeneous $C^*$-algebra} will be a
$C^*$-algebra isomorphic to a $C^*$-algebra of the form
$eC(X,M_l)e$ where $X$ is a compact metric space and $e$ is a projection in
$C(X,M_l)$ such that $e(x) \neq 0$ for all $x \in X$.

\begin{defn}\label{slowdim} A unital $C^*$-algebra $A$ is an \emph{AH-algebra} when
  there is an increasing sequence $A_1 \subseteq A_2 \subseteq A_3
  \subseteq \dots$ of unital $C^*$-subalgebras of $A$ such that $A =
  \overline{\bigcup_n A_n}$ and each $A_n$ is a homogeneous
  $C^*$-algebra. We say that $A$ has \emph{no dimension
    growth} when the sequence $\{A_n\}$ can be chosen such that 
$$
A_n \simeq e_nC\left(X_n,M_{l_n}\right)e_n
$$
with $\sup_n \Dim X_n < \infty$ and $\lim_{n \to \infty} \min_{x \in
  X_n} \Rank e_n(x) = \infty$.
\end{defn}

Note that the no dimension growth condition is stronger than the slow
dimension growth condition used in \cite{Th3}.

\begin{prop}\label{AHthm} Assume that $\Dim Y < \infty$ and that
  $\phi$ is strongly transitive and not injective. It follows
  $C^*_r\left(R_{\phi}\right)$ is an AH-algebra with no dimension growth. 
\end{prop}
\begin{proof} For each $n$ we have that 
\begin{equation}\label{renu}
 C^*_r\left(R\left(\phi^n\right)\right) \simeq e_nC\left(Y,M_{m_n}\right)e_n
\end{equation}
for some $m_n \in \mathbb N$ and some projection $e_n \in
C\left(Y,M_{m_n}\right)$. Although this seems to be well known it is
hard to find a proof anywhere so we point out that it can proved by
specializing the proof of Theorem 3.2 in \cite{Th1} to the case of a
surjective local homeomorphism $\phi$. In fact, it suffices to observe
that the $C^*$-algebra $A_{\phi}$ which features in Theorem 3.2 of
\cite{Th1} is $C(Y)$ in this case. Since $\min_{y \in Y} \Rank e_n(y)$
is the minimal dimension of an irreducible representation of
$C^*_r\left(R\left(\phi^n\right)\right)$ it therefore now suffices to show that the minimal dimension of the irreducible
  representations of $C^*_r\left(R(\phi^n)\right)$ goes to infinity when
  $n$ does. It follows from Lemma 3.4 of \cite{Th1} that the minimal dimension of the irreducible
  representations of $C^*_r\left(R(\phi^n)\right)$ is the same as the
  number
$\min_{y \in Y} \# \phi^{-n}(y)$. It follows from Lemma \ref{hmzero}
that
$$
\lim_{n \to \infty} \min_{y \in Y} \# \phi^{-n}(y) = \infty ,
$$
exponentially fast in fact.
\end{proof}

\begin{lemma}\label{?1} Assume that $C^*_r\left(\Gamma_{\phi }\right)$
  is simple. Then either $\phi $ is a homeomorphism or else 
\begin{equation}\label{limit}
\lim_{n \to \infty} \sup_{x \in Y}
m(x)^{-1}m(\phi (x))^{-1}m\left(\phi ^2(x)\right)^{-1} \dots
m\left(\phi ^{n-1}(x)\right)^{-1} = 0 ,
\end{equation}
where $m : Y \to \mathbb N$ is the function (\ref{m-funk}).
\begin{proof} Assume (\ref{limit}) does not hold. Since $\phi$ is a
  local homeomorphism, the function $m$ is continuous so it follows
  from Dini's theorem that
  there is at least one $x$ for which  
\begin{equation}\label{limit2}
\lim_{n \to \infty}
m(x)^{-1}m(\phi (x))^{-1}m\left(\phi ^2(x)\right)^{-1} \dots
m\left(\phi ^{n-1}(x)\right)^{-1} 
\end{equation}
is not zero. For this $x$ there is a $K$ such that $\# \phi^{-1}
\left(\phi^k(x)\right) = 1$ when $k \geq K$, whence the set 
$$
F = \left\{ y \in
  Y : \ \# \phi ^{-1}\left( \phi^k(y)\right) =1 \ \forall k \geq 0\right\}
$$
is not empty. Note that $F$ is closed and that $\phi ^{-k}\left(\phi ^k(F)\right) =
F$ for all $k$, i.e. $F$ is $\phi$-saturated. It follows from
Corollary \ref{A5} that $F$ determines a proper 
ideal $I_F$ in
$C_r^*(R_\phi)$. Since $\phi(F) \subseteq F$, it follows that
$\widehat{\phi}(I_F)\subseteq I_F$.  Then Theorem 4.10 of
\cite{Th1} and the 
simplicity of $C^*_r\left(\Gamma_{\phi}\right)$ imply that either $\phi$ is
injective or $I_F = \{0\}$. But $I_F = \{0\}$ means that $F=Y$ and
thus that $\phi$ is
injective. Hence $\phi$ is a homeormophism in both cases.
\end{proof}
\end{lemma}

\begin{thm}\label{quotients} Let $\varphi : X \to X$ be a locally
  injective surjection on a compact metric
  space $X$ of finite covering dimension, and let $(Y,\phi)$ be its canonical
  locally homeomorphic extension. Let $A$ be a simple quotient of
  $C^*_r\left(\Gamma_{\varphi}\right)$. It follows that $A$ is
  $*$-isomorphic to either
\begin{enumerate}
\item[1)] a full matrix algebra $M_n(\mathbb C)$ for some $n \in
  \mathbb N$, or
\item[2)] the crossed product $C(F) \times_{\phi|_F} \mathbb Z$
  corresponding to an infinite minimal closed totally $\phi$-invariant
  subset $F \subseteq Y$ on which $\phi$ is injective, or
\item[3)] a purely infinite, simple, nuclear, separable $C^*$-algebra; more specifically to
  the crossed product $C^*_r\left(R_{\phi|_F}\right)
  \times_{\widehat{\phi|_F}} \mathbb N$ where $F$ is an infinite
  minimal closed totally $\phi$-invariant subset of $Y$ and
  $C^*_r\left(R_{\phi|_F}\right)$ is an AH-algebra with no dimension growth.
\end{enumerate}
\begin{proof} If $A$ is not a matrix algebra it has the form
  $C^*_r\left(\Gamma_{\phi|_F}\right)$ for some infinite minimal closed totally $\phi$-invariant
  subset $F \subseteq Y$ by (\ref{basiciso}) and Corollary \ref{maximal3}. If $\phi$ is injective on $F$ we are in case
  2). Assume not. Since $\Dim F \leq \Dim Y \leq \Dim X$ by
  Proposition \ref{dim!!!} it follows from Proposition \ref{AHthm}
  that $C^*_r\left(R_{\phi|_F}\right)$ is an AH-algebra with no
dimension growth. By \cite{An} (or Theorem 4.6 of \cite{Th1}) we have an isomorphism
$$ 
C^*_r\left(\Gamma_{\phi|_F}\right) \simeq C^*_r\left(R_{\phi|_F}\right) \times_{\widehat{\phi|_F}} \mathbb N,
$$
where $\widehat{\phi|_F}$ is the endomorphism of
$C^*_r\left(R_{\phi|_F}\right)$ given by conjugation with
$V_{\phi|_F}$. We claim that the pure infiniteness of $
C^*_r\left(R_{\phi|_F}\right) \times_{\widehat{\phi|_F}} \mathbb N$
follows from Theorem 1.1 of \cite{Th3}. For this it remains only to check that
$\widehat{\phi|_F} = \Ad V_{\phi|_F}$ satisfies the two conditions on
$\beta$ in Theorem 1.1 of \cite{Th3}, i.e. that $\widehat{\phi|_F}(1) =
V_{\phi|_F}V_{\phi|_F}^*$ is a full projection and that there is no
$\widehat{\phi|_F}$-invariant trace state on
$C^*_r\left(R_{\phi|_F}\right)$. The first thing was observed already
in Lemma 4.7 of \cite{Th1} so we focus on the second. Observe that it follows from
Lemma 2.24 of \cite{Th1} that
$\omega = \omega \circ P_{R_{\phi}}$ for every trace state $\omega$ of
$C^*_r\left(R_{\phi}\right)$.  By using this, a direct
calculation as on page 787 of \cite{Th1} shows that
$$
\omega\left( V_{\phi|_F}^n{V_{\phi|_F}^*}^n\right) \leq \sup_{y \in Y}
\left[m(y)m(\phi(y))\dots m\left(\phi^{n-1}(y)\right)\right]^{-1}
$$
Then Lemma \ref{?1} implies that $\lim_{n \to \infty} \omega\left(
  V_{\phi|_F}^n{V_{\phi|_F}^*}^n\right) = 0$. In particlar, $\omega$ is not $\widehat{\phi|_F}$-invariant.
\end{proof}
\end{thm}

\begin{cor}\label{cor1} Assume that
  $C^*_r\left(\Gamma_{\varphi}\right)$ is simple and that $\Dim X < \infty$. It follows that
  $C^*_r\left(\Gamma_{\varphi}\right)$ is purely infinite if and only if
  $\varphi$ is not injective.
\end{cor} 
 \begin{proof} Assume first that $\varphi$ is injective. Then
   $C^*_r\left(\Gamma_{\varphi}\right)$ is the crossed product $C(X)
   \times_{\varphi} \mathbb Z$ which is stably finite and thus not
   purely infinite. 

   Conversely, assume that $\varphi$ is not
   injective. Then a direct calculation, as in the proof of Theorem 4.8
   in \cite{Th1}, shows that $V_{\varphi}$ is a
   non-unitary isometry in $C^*_r\left(\Gamma_{\varphi}\right)$. 
   Since the $C^*$-algebras which feature in case 1) and case 2) of 
   Theorem \ref{quotients} are stably finite, the presence of a non-unitary isometry
   implies that $C^*_r\left(\Gamma_{\varphi}\right)$ is purely infinite.
\end{proof}

\begin{cor}\label{tokecor}
  Let $S$ be a one-sided subshift. If the $C^*$-algebra
  $\mathcal{O}_S$ associated with $S$ in \cite{C} is simple, then it is also
  purely infinite.
\end{cor}

\begin{proof}
  It follows from Theorem 4.18 in \cite{Th1} that $\mathcal{O}_S$ is
  isomorphic to $C^*_r\left(\Gamma_{\sigma}\right)$ where $\sigma$ is
  the shift map on $S$.
  If $\mathcal{O}_S$ is simple, $S$ must be infinite and it then
  follows from Proposition 2.4.1 in \cite{BS} (cf. Theorem 3.9 in
  \cite{BL}) that $\sigma$ is not injective. The
  conclusion follows then from Corollary \ref{cor1}.
\end{proof}

In Corollary \ref{tokecor} we assume that the shift map
$\sigma$ on $S$ is surjective. It is not clear if the result holds
without this assumption.

For completeness we point out that when $X$ is totally disconnected
(i.e. zero dimensional) the algebra $C^*_r\left(R_{\phi|_F}\right)$
which features in case 3) of Theorem \ref{quotients} is approximately
divisible, cf. \cite{BKR}. We don't know if this is the case in general,
but a weak form of divisibility is always present in
$C^*_r\left(R_{\phi}\right)$ when $C^*_r\left(\Gamma_{\varphi}\right)$
is simple and $\phi$ not injective, cf. \cite{Th3}.

\begin{prop}\label{propfinish} Assume that $Y$ is totally disconnected
  and $\phi$ strongly transitive and not injective. It follows that
  $C^*_r\left(R_{\phi}\right)$ is an approximately divisible AF-algebra.
\begin{proof} It follows from Proposition 6.8 of \cite{DS} that $C^*_r\left(R_{\phi}\right)$ is an
  AF-algebra. As pointed out in Proposition 4.1 of \cite{BKR} a unital
  AF-algebra fails to be approximately divisible only if it has a
  quotient with a non-zero abelian projection. If $C^*_r\left(R_{\phi}\right)$ has such a quotient there is also a primitive
  quotient with an abelian projection; i.e. by Proposition \ref{A17} there is
  an $x \in Y$ such that $C^*_r
  \left(R_{\phi|_{\overline{H(x)}}}\right)$ has a non-zero abelian
  projection $p$. It follows from (\ref{crux}) that every projection
  of $C^*_r
  \left(R_{\phi|_{\overline{H(x)}}}\right)$ is unitarily equivalent to
  a projection in $C^*_r
  \left(R\left(\phi^n|_{\overline{H(x)}}\right)\right)$ for some
$n$. Since $\overline{H(x)}$ is totally disconnected we can use Proposition 6.1 of
  \cite{DS} to conclude that every projection in $C^*_r
  \left(R\left(\phi^n|_{\overline{H(x)}}\right)\right)$ is unitarily equivalent to
  a projection in $D_{R_{\phi|_{\overline{H(x)}}}} =
  C\left(\overline{H(x)}\right)$. We may therefore assume that $p \in
  C\left(\overline{H(x)}\right)$ so that $p = 1_A$ for some clopen $A
  \subseteq \overline{H(x)}$. Then $H(x) \cap A \neq \emptyset$ so by
  exchanging $x$ with some element in $H(x)$ we may assume that $x \in
  A$. If there is a $y \neq x$ in $A$ such that $\phi^k(x) =
  \phi^k(y)$ for some $k \in \mathbb N$, consider functions $g \in
  C\left(\overline{H(x)}\right)$ and $f \in C_c\left(R_{\phi}\right)$
  such that $g(x) = 1, g(y) = 0$, $\supp g \subseteq A$, $\supp f \subseteq R_{\phi} \cap
  \left(A \times A\right)$ and $f(x,y) \neq 0$. Then $f,g \in 1_AC^*_r
  \left(R_{\phi|_{\overline{H(x)}}}\right)1_A$ and $gf \neq 0$
  while $fg = 0$, contradicting that
  $1_AC^*_r\left(R_{\phi|_{\overline{H(x)}}}\right)1_A$ is abelian. Thus no
  such $y$ can exist which implies that $\pi_x(1_A) = 1_{\{x\}}$, where
  $\pi_x$ is the representation (\ref{pirep}), restricted to the
  subspace of $H_x$ consisting of the functions supported in
  $\left\{(x',k,x) \in \Gamma_{\phi} : \ k = 0 \right\}$. It follows
  that
  $\pi_x\left(1_AC^*_r\left(R_{\phi|_{\overline{H(x)}}}\right)1_A\right) \simeq \mathbb C$. Consider a non-zero ideal $J \subseteq \pi_x\left(C^*_r\left(R_{\phi|_{\overline{H(x)}}}\right)\right)$. Then $\pi_x^{-1}(J)$ is a non-zero ideal in $C^*_r\left(R_{\phi|_{\overline{H(x)}}}\right)$ and it follows from Corollary \ref{A5} that there is an open non-empty subset $U$ of $\overline{H(x)}$ such that $\phi^{-k}\left( \phi^k(U)\right) = U$ for all $k$ and $C_0(U) = \pi_x^{-1}(J) \cap C\left(\overline{H(x)}\right)$. Since $H(x) \cap U \neq \emptyset$, it follows that $x \in U$ so there is a function $g \in \pi_x^{-1}(J) \cap C\left(\overline{H(x)}\right)$ such that $g(x) = 1$. It follows that $\pi_x\left(g1_A\right) = 1_{\left\{x\right\}} = \pi_x\left(1_A\right) \in J$. This shows that $\pi_x(1_A)$ is a full projection in $\pi_x\left(C^*_r\left(R_{\phi|_{\overline{H(x)}}}\right)\right)$ and Brown's theorem, \cite{Br}, shows now that $\pi_x\left(C^*_r\left(R_{\phi|_{\overline{H(x)}}}\right)\right)$ is stably isomorphic to $\pi_x\left(1_AC^*_r\left(R_{\phi|_{\overline{H(x)}}}\right)1_A\right) \simeq \mathbb C$. Since $\pi_x\left(C^*_r\left(R_{\phi|_{\overline{H(x)}}}\right)\right)$ is unital this means that it is a full matrix algebra. In conclusion we deduce that if $C^*_r\left(R_{\phi}\right)$ is not approximately divisible it has a full matrix
  algebra as a quotient. By Corollary \ref{A5} this implies that
there is a finite set $F'
  \subseteq Y$ such that $F' = \phi^{-k}\left(\phi^k(F')\right)$
  for all $k \in \mathbb N$. Since 
$$
\phi^{-k}\left(\phi^k(x)\right) \subseteq
\phi^{-k-1}\left(\phi^{k+1}(x)\right) \subseteq F'
$$ for all $k$ when $x \in F'$, there is for each $x \in F $ a natural
number $K$ such that $\phi^{-k}\left(\phi^k(x)\right) =
\phi^{-K}\left(\phi^K(x)\right)$ when $k \geq K$.  Then $\# \phi^{-1} \left(
  \phi^k(x)\right) =1$ for $k \geq K+1$, so that $m\left(\phi^k(x)\right)
= 1$ for all $k \geq K$, which by Lemma \ref{hmzero} contradicts that
$\phi$ is not injective. This
contradiction finally shows that $C^*_r\left(R_{\phi}\right)$ is
approximately divisible, as desired.
\end{proof}
\end{prop}

\bibliographystyle{plain}


\end{document}